\documentclass{article}%

\usepackage[title]{appendix}
\usepackage{setspace}

\usepackage{lipsum}
\usepackage{amsfonts}
\usepackage{amsthm}
\usepackage{graphicx}
\usepackage{epstopdf}
\usepackage{algorithmic}
\usepackage{amsopn}
\usepackage{tikz-cd}
\usepackage{enumerate}
\usepackage{mathtools}
\usepackage{cleveref}

\theoremstyle{plain}
\newtheorem{theorem}{Theorem}[section] 
\newtheorem{lemma}[theorem]{Lemma}
\newtheorem{proposition}[theorem]{Proposition}
\newtheorem{corollary}[theorem]{Corollary}

\theoremstyle{definition}
\newtheorem{definition}[theorem]{Definition}

\theoremstyle{remark}

\newtheorem{property}[theorem]{Property}
\newtheorem{assumption}[theorem]{Assumption}
\newtheorem{example}[theorem]{Example}

\crefname{theorem}{Theorem}{Theorems}
\crefname{example}{Example}{Examples}
\crefname{lemma}{Lemma}{Lemmata}
\crefname{proposition}{Proposition}{Propositions}
\crefname{corollary}{Corollary}{Corollaries}
\crefname{equation}{}{}
\crefname{property}{Property}{Properties}
\crefname{assumption}{Assumption}{Assumptions}
\crefname{definition}{Definition}{Definitions}

\newcommand{\ed}{\mathrm{d}}

\newcommand{\ol}{\overline}

\begin{document}

\title{Curved Schemes for SDEs on Manifolds}

\author{
John Armstrong and Tim King}
\date{}
\maketitle

\begin{abstract}
  Given a stochastic differential equation (SDE) in $\mathbb{R}^n$ whose solution is constrained to lie in some manifold $M \subset \mathbb{R}^n$, we propose a class of numerical schemes for the SDE whose iterates remain close to $M$ to high order. Our schemes are geometrically invariant, and can be chosen to give perfect solutions for any SDE which is diffeomorphic to $n$-dimensional 
  Brownian motion. Unlike projection-based methods, our schemes may be  implemented without explicit knowledge of M. Our approach does not require simulating any iterated It\^{o} interals beyond those needed to implement the Euler--Maryuama scheme. We prove that the schemes converge under a standard set of assumptions, and illustrate their practical advantages by considering a stochastic version of the Kepler problem.
\end{abstract}

\numberwithin{equation}{section}
\section{Introduction}
When studying the dynamics of a complex system, there are often constraints of a geometric nature. For example, to say that a particle moving in phase space has constant energy is to specify a manifold on which the particle must lie. Other examples arise in control theory, for instance the movement of a robot arm, constrained to have constant length. Therefore, if applying a numerical method to predict the movement of such a system, it seems sensible to choose a method which respects the underlying geometry.

We are interested in modelling systems subject to the influence of random noise. Specifically, we study stochastic differential equations (SDEs) in $\mathbb{R}^n$ driven by $k$-dimensional Brownian motion. We assume that there is some manifold
$M \subset \mathbb{R}^n$ to which the solution is confined, and aim to find a numerical scheme for this SDE which remains close to $M$ for long periods of time.

One way to do this, given some numerical scheme, is to successively project each of its iterates onto $M$ before finding the next iterate. This is explored for ODEs in \cite{hairer2006geometric} and for SDEs in \cite{averina2019modification}. In the latter case, the authors show that for a large class of schemes, introducing the projection step does not adversely affect the local order of the scheme. However, to compute the projection requires finding the intersection of $M$ with a straight line passing through the iterate, and this in general means solving a nonlinear
system of equations, which might not have a unique (or any) solution. More fundamentally, it requires knowing the equation of $M$ in the first place.

We propose a class of numerical methods, called \textit{jet schemes}, which use differential geometric flow, rather than projection. Our method is independent of the coordinates in which the SDE is expressed, in a sense we make precise in \Cref{sec:invariance}. As a result, implementing the method does not require knowing the equation of $M$. In addition, the coordinate invariance means that the jet schemes should perform well on problems that, upon making a judicious choice of coordinates, take a simple form, such as Brownian motion or additive noise. By continuity of geometric flow, we expect good performance on problems which are small perturbations of this type, and further it is possible to use the schemes if the solution trajectories are merely concentrated near $M$, as opposed to being strictly confined to it. The judicious choice of coordinates need not be known explicitly, and so the jet schemes may be able to detect conserved or nearly-conserved quantities, even if they are not known a priori.

The most well known scheme for SDEs is the Euler--Maruyama 
scheme. This scheme works by following short straight line segments, and hence leaves $M$ rather quickly if $M$ is curved. It would be better to follow short curved segments in $M$ instead, which motivates us to look at geometric flows on $M$. This requires solving an ODE, but there are a wealth of high order methods to do this, many of which are not significantly more expensive than the Euler scheme.

By contrast, higher order schemes for SDEs, such as the Milstein scheme, require  (see \cite{rumelin1982numerical}) the simulation of iterated It\^{o} integrals of the form
\[
\int_0^{t_1} \int_0^{t_2} \ldots \int_0^{t_n} \ed W^{\alpha_1}_{s} \ed W^{\alpha_2}_{t_{n}} \ldots \ed W^{\alpha_n}_{t_{2}} 
\qquad \alpha_1, \alpha_2,\ldots, \alpha_n \in \{1, \ldots k\}.
\]
In more than one dimension this is nontrivial to do, and adds an extra layer of complexity. Although methods have been developed for doing so (e.g using Fourier series
\cite{kuznetsov2017expansion} or Hermite polynomials
\cite{kloeden2013numerical}), there is still no geometric reason to suppose that the solution will remain close to $M$, beyond the fact that the numerical approximation converges to the true solution. Jet schemes do not require the simulation of such integrals.

The idea of using geometry to inform the development of numerical methods is somewhat recent; according to
\cite{iserles2018geometric}, `the importance of [geometric numerical integration] has been recognised and its scope delineated only in the 1990s.' A substantial part of this
work centred on systems that have a \textit{Hamiltonian} structure. So-called symplectic ODE methods have been developed  which preserve the quantities which are naturally conserved in Hamiltonian systems \cite{hairer2006geometric}. They have subsequently been extended to Hamiltonian SDE systems \cite{milstein2002numerical}
\cite{milstein2002symplectic}.
More akin to our work is methods for differential equations on Lie groups $M$. Here, one translates an equation on $M$ to a corresponding equation on the Lie algebra $\mathcal{M}$, a linear space, and uses a kind of flow - the exponential map - to get back to $M$. These ideas have since been used to modify ODE methods so that they stay in $M$. The authors in
\cite{malham2008stochastic} play a similar game for SDEs. One 
technique used in \cite{malham2008stochastic} is an expansion involving successive iterations of the commutator bracket, called the \textit{Magnus expansion}. Use of the Magnus expansion can give methods which are superior to those using the stochastic Taylor expansion \cite{lord2008efficient}. Moreover, by choosing when to truncate the expansion, one can obtain methods of strong order greater than $\frac{1}{2}$. Our scheme (in its current form) does not have these advantages, but it works on any smooth manifold, not just those that have Lie group structure.

The paper is organized as follows. Our model and main results are in \Cref{sec:modelandmainresult}. \Cref{sec:howtochoosegamma} discusses when the assumptions of the model hold. \Cref{sec:invariance} discusses the geometric invariance of our scheme, and \Cref{sec:examples} applies the scheme to a stochastic version of the Kepler problem. \Cref{sec:workhorse}
contains a general result on the convergence of numerical schemes, and \Cref{sec:strongconvergence,sec:weakconvergence} apply this general result to prove our main result in the strong and weak senses respectively. The appendix contains some of the longer or less interesting proofs, included for completeness.

\section{Model and Main Result}
\label{sec:modelandmainresult}
We aim to simulate the process $X=(X_t)_{t \in [0,T]}$ in $\mathbb{R}^d$ given by the It\^{o} SDE 
\begin{equation}
\ed X_t = a(X_t,t) \, \ed t + \sum_{i=1}^k b_\alpha(X_t,t) \, \ed W_t^\alpha
\label{eq:sdeIto}
\end{equation}
where $[0,T]$ is the time interval, the $(W_t^\alpha)_{\alpha=1,2,\dots,k}$ are a collection of $k$ independent (one dimensional) Brownian motions, $a(x,t):\mathbb{R}^n \times [0,T] \to \mathbb{R}^n$ and $b_{\alpha}(x,t):\mathbb{R}^n \times [0,T] \to \mathbb{R}^n$ for each $\alpha$. Our initial condition is $X_0=x_0$ for some fixed $x_0 \in \mathbb{R}^n$. The data above may be equivalently expressed in Stratonovich form as
\begin{equation}
\ed X_t = \ol{a}(X_t,t) \, \ed t + \sum_{i=1}^k b_\alpha(X_t,t) \, \circ \ed W_t^\alpha
\label{eq:sdeStratonovich}
\end{equation}
where
\begin{equation}
\ol{a}(x,t):=a(x,t) - 
\frac{1}{2} \sum_{\alpha=1}^{k}\sum_{j=1}^n b_\alpha(x,t)^j
\frac{\partial b_\alpha}{\partial x^j}.
\end{equation}
We make the following assumptions on the SDE. 
\begin{assumption}
	$a(x,t)$ and $b(x,t)$ are Lebesgue measurable.
	\label{assumptionMeasurable}
\end{assumption}
\begin{assumption}
	$a$ and $b$ are uniformly Lipschitz in $x$. That is, there exists a constant $K>0$ such that
	\[
	|a(x,t)-a(y,t)|<K|x-y| \text{ and } |b(x,t)-b(y,t)|<K|x-y|
	\]
	for all $x,y \in \mathbb{R}^d,t \in [0,T]$.
	\label{assumptionLipschitz}
\end{assumption}
\begin{assumption}
	There is a manifold $M \subset \mathbb{R}^n$, containing $x_0$, such that for all $x \in M$, $\ol{a}(x,t)$ and $b(x,t)$ are tangent to $M$.
	\label{assumptionTangent}
\end{assumption}

\cref{assumptionMeasurable,assumptionLipschitz} ensure that \cref{eq:sdeIto} (or equivalently \cref{eq:sdeStratonovich}) has a unique solution. As is well known, \cref{assumptionTangent} implies that the true solution of \eqref{eq:sdeIto} remains in $M$. Of course, \cref{assumptionTangent} trivially holds when $M=\mathbb{R}^n$, but we are interested in the case where the inclusion is proper.

For each point $x \in \mathbb{R}^n$, choose a map \begin{equation*} \gamma(x,t,v): \mathbb{R}^n \times [0,T] \times (\mathbb{R} \times \mathbb{R}^k) \rightarrow \mathbb{R}^n, \end{equation*} where the parameter $v=(v_0,v_1,\dots,v_k)$ may be thought of as $(\delta t, \delta W^1, \dots,\delta W^k)$. For fixed values of $v,t$, the map $\gamma$ may be visualised as a field of curves in the ambient space $\mathbb{R}^n$ whose definition does not require any information about the driving Brownian motion beyond that required for the Euler--Maruyama scheme. In practice, it may only be possible to compute $\gamma$ approximately; let $\tilde{\gamma}(x,t,v)$ denote this approximation.

\begin{definition}
The \textit{jet scheme} (for a particular choice of $\gamma$) is the following numerical scheme for the solution $X$ of \cref{eq:sdeIto}. Let a discretisation $0=t_0<t_1<\dots<t_N=T$ be given. Set $Y^{N, \gamma}_0 = x_0$, and put
\begin{equation}
\label{eq:jetscheme} Y_{i+1}^{N,\gamma} = \tilde{\gamma}(Y_i,t,\delta t_i, \delta W_i)
\end{equation}
where $\delta t_i = t_{i+1} - t_{i}$, $\delta W_i = W_{t_{i+1}} - W_{t_i}$. Thus $Y_{i}$ is an approximation to $X_{t_i}$.
\end{definition}

In the case where $\gamma$ is sufficiently regular, the approximation $\tilde{\gamma}=\gamma$ is perfect, and the time discretisation is evenly spaced, it is shown in \cite{armstrong2018coordinate} that $Y$ will converge in the $L^2$ sense, as $N \rightarrow \infty$, to the solution of the It\^o SDE 
\[ \ed X_t = \frac{1}{2}(\Delta \gamma_{X,t})(0) \, \ed t + (\nabla_{\ed W_t} \gamma_{X,t})(0). \]
Here $\Delta$ is the Laplacian and $\nabla$ is the Euclidean covariant derivative on $\mathbb{R}^k$. In particular, the limit of \cref{eq:jetscheme} as $\delta t \to 0$ only depends upon the first- and second- order derivatives of $\gamma_{x,t}$. In the language of differential geometry, we say that the limit is determined by the 2-jet of $\gamma$, hence the name of our scheme. The point is that the $Y_i$ will converge to the SDE of interest provided that we can arrange for $\gamma$ to have the correct first- and second- order derivatives, and appropriate regularity. We therefore assume that $\gamma$ has the following properties.

\begin{property}(Correct 2-jet)
	When $v=0$, we have
	\[ \gamma(x,t,0)=x, \]
	\[ \frac{\partial \gamma(x,t,0)}{\partial v^\alpha}=b_{\alpha}(x,t), \]
	\[ \frac{\partial \gamma(x,t,0)}{\partial v^0} + \frac{1}{2}\sum_{\alpha=1}^k \frac{\partial^2 \gamma(x,t,0,0)}{\partial v^\alpha \partial v^\alpha}=a(x,t). \]
	\label{propertyFirstTwoDerivs}
\end{property}

\begin{property}(Remains in $M$)
	If $x \in M$ then $\gamma(x,t,v) \in M$ for all $t \in [0,T]$, $v \in \mathbb{R}^{k+1}$.
	\label{propertyRemainsInM}
\end{property}

\begin{property}(Growth of Derivatives)
	For this property, we first specify some integer $r$. The property 	holds for this value of $r$ if, for each multi-index $\alpha$ with $|\alpha| \leq r$, we have 
	\[ \left| \partial_{v}^\alpha \gamma(x,t,v)  \right| \leq K(1+|x|)e^{K(1+|v|^2)} \]
	for some constant $K$ which may depend on $\alpha$ but not on $x$ or $v$.
	\label{propertyrthDerivative}
\end{property} 

\begin{property} (Lipschitz derivatives) For each multi-index $\alpha$ with $|\alpha| \leq 2$, we have
	\[ \left| \partial^\alpha_v \gamma(x,t,v) - \partial^\alpha_v \gamma(y,t,v) \right|
	\leq K|x-y|e^{K(1+|v|^2)} \]
	for some constant $K$ which may depend on $\alpha$ but not on $x$, $y$, or $v$.
	\label{propertyLipschitzDerivs}
\end{property}

These assumptions are more general than in \cite{armstrong2018coordinate}, which, for example, required the derivatives to be globally bounded, and did not allow $\gamma$ to depend explicitly on $v_0$.

In this paper, $\gamma$ will be defined implicitly by an ODE, and $\tilde{\gamma}$ produced by a user-chosen ODE scheme. The following definition measures the accuracy of this scheme.

\begin{definition}
Let $m \in \mathbb{N}$. We say that $\tilde{\gamma}(x,t,v)$ is an $m$-good approximation to $\gamma(x,t,v)$ if there exists a constant $K$ such that 
\[ |\tilde{\gamma}(x,t,v)-\gamma(x,t,v)| \leq K |v|^m e^{K (1+|v|^2)} (1+|x|) \]
for all $x,t,v$.
\end{definition}

The main force of this definition is that, when $|v|$ is small, we have
\[ |\tilde{\gamma}(x,t,v)-\gamma(x,t,v)| = O(|v|^m (1+|x|)). \]
The value of $m$ depends upon the choice of ODE scheme. If $|v|$ is small, then typically only a single time step will be required for the ODE scheme to achieve the desired rate of convergence. Given that the bound holds for small $|v|$, the assumption that it holds for large $|v|$ is not at all strong and should be true for any reasonable ODE scheme based on Taylor's theorem; see \Cref{sec:howtochoosegamma} for further discussion.

An obvious way to define $\tilde{\gamma}$ is to truncate the Taylor series for $\gamma$, obtaining 
\begin{equation}
\tilde{\gamma}_{(r)}(x,v,t)=x+\sum_{i=1}^r \sum_{|\alpha| = i} \frac{v^\alpha}{\alpha!} \partial_v^\alpha \gamma(x,t,0). \label{eq:orderrexpansionjetscheme}
\end{equation}

For an integer $r$, we define the \textbf{order $r$ expansion jet scheme} to be the jet scheme, choosing $\tilde{\gamma}$ as in (\ref{eq:orderrexpansionjetscheme}). In this case, the ODE-solving part of the scheme is explicit and requires only a single time step, as promised. It will be useful not only as a practical scheme but also in proving the convergence properties of other jet schemes.

Before stating our main result, we define the two main notions of convergence for SDE schemes.

\begin{definition}
Let $0=t_0 < t_1 < \dots < t_N = T$ be a given discretisation of $[0,T]$. Let $(X_t)_{t \in [0,T]}$ be the true solution of \eqref{eq:sdeIto}, and let $Y_t$ be an approximation to $X_t$ obtained from some numerical scheme. Let $\delta t = \max_{0 \leq i < N}(t_{i+1}-t_{i})$. We say that the scheme \textit{converges with strong order $n$} if 
\begin{equation}
\mathbb{E}[\max_i |Y_{t_i} - X_{t_i}|^2] \leq C(\delta t)^{2n}
\end{equation}
for some constant $C$ independent of the discretisation. We say that the scheme \textit{converges with weak order $n$} if, for any smooth function $g: \mathbb{R}^n \rightarrow \mathbb{R}$ of at most polynomial growth,
\begin{equation}
\left| \: \mathbb{E}[g(Y(t_N)) - g(X(t_N))] \: \right| \leq C(\delta t)^n
\end{equation} 
where $C$ is again independent of the discretisation.
\end{definition}

The concepts of strong and weak convergence are distinct. High strong order signifies good approximation of the path $t \mapsto X_t$, whilst high weak order signifies good approximation of integrals such as $\mathbb{E}[X_T]$.

\begin{theorem} Suppose that $\gamma$ satisfies \cref{propertyFirstTwoDerivs,propertyrthDerivative} for $r=2$. Let $\tilde{\gamma}$ be a $3$-good approximation to $\gamma$. Then the jet scheme converges to the solution of \cref{eq:sdeIto} with strong order $\frac{1}{2}$ and weak order 1. Suppose in addition that $\gamma$ satisfies \cref{propertyRemainsInM,propertyLipschitzDerivs}, and that $\tilde{\gamma}$ is $m$-good for some $m \geq 3$. Then 
\begin{equation}
\label{eq:mainresstrong}
\mathbb{E}[ \max_{1 \leq i \leq N}\inf_{x \in M}\{ |Y^{\tilde{\gamma}}_{i} - x|^2 \} ] = O( (\delta t)^{m-2} ).
\end{equation} where $(Y^{\tilde{\gamma}}_{i})_{i=1}^N$ denotes the iterates of the approximate jet scheme. Also, for any smooth function $g: \mathbb{R}^n \rightarrow \mathbb{R}$ of at most polynomial growth we have
\begin{equation}
\label{eq:mainresweak}
|\mathbb{E}[\inf_{x \in M}\{g(Y^{\tilde{\gamma}}_i)-g(x) \}]| =O((\delta t)^{\frac{m}{2}}).
\end{equation}
\label{thm:convergenceOnManifolds}
\end{theorem}

\Cref{thm:convergenceOnManifolds} is our main result, the significance of which is as follows. The Euler--Maruyama scheme converges with strong order $\frac{1}{2}$ and weak order $1$, so the jet scheme converges just as well as the Euler--Maruyama scheme. However, in the sense of remaining close to $M$, the jet scheme resembles a higher-order scheme with strong order $\frac{m}{2}-1$ and weak order $\frac{m}{2}$. Physically, the high strong order convergence means that the sample paths of the scheme, even if not perfectly accurate, will nevertheless lie close to $M$. The high weak-order convergence means that the scheme will, for example, estimate the expected energy of the system to a high degree of accuracy. 

\section{Choosing the jet map $\gamma$} \label{sec:howtochoosegamma}

To use jet schemes in practice, we need a map $\gamma$ (and $\tilde{\gamma}$) satisfying the given properties, and which is easy to calculate. Recall that $\ol{a}$ was defined in \cref{eq:sdeStratonovich} to be the Stratonovich analog of $a$.

\begin{lemma}
Given $x,v \in \mathbb{R}^n$, $t \in [0,T]$, let $X(x,t,v)$ be the vector field
\begin{equation}
X(x,t,v) = \sum_{\alpha=1}^k v^{\alpha}b_{\alpha}(x,t) + \frac{1}{k} \sum_{\alpha=1}^k (v^{\alpha})^2 \: \ol{a}(x,t).
\label{eq:PhiDefinition}
\end{equation}
Given $s \in [0,1]$, let $\Phi(x,t,v,s)$ be the flow of $X(x,t,v)$, starting from $x$ and running for time $s$. This means that $\Phi$ satisfies the differential equation 
\[ \frac{\partial{\Phi}}{\partial s}=X(\Phi(x,t,v),t,v); \hspace{10pt} \Phi(x,t,v,0)=x. \]
Finally, set $\gamma(x,t,v)=\Phi(x,t,v,1)$. Then \cref{propertyFirstTwoDerivs,propertyRemainsInM} are both satisfied.
\label{lemma:gammadef}
Moreover, if we instead define $X$ to be 
\begin{equation} 
X(x,t,v) = \sum_{\alpha=1}^k v^{\alpha}b_{\alpha}(x,t) + v^0 \: \ol{a}(x,t).
\label{eq:PhiDefinition2}
\end{equation} and proceed as before, then the same conclusion holds.
\end{lemma}
The proof of this lemma is a direct computation; see the appendix for details.

We see that \cref{propertyFirstTwoDerivs,propertyRemainsInM} do not specify $\gamma$, uniquely; the optimal choice of $\gamma$ is likely to depend on the problem at hand. We shall call the schemes obtained by choosing $\gamma$ as in \cref{eq:PhiDefinition,eq:PhiDefinition2} the $(\delta W)^2$-jet scheme and the $(\delta t)$-jet scheme respectively. The $(\delta W)^2$-jet scheme has no explicit dependence on $v_0$, so may be easier to analyse theoretically. On the other hand, the $(\delta t)$-jet scheme can give perfect answers in cases where the $(\delta W)^2$-jet scheme does not - see \Cref{sec:invariance}. As observed in \cite{armstrong2018coordinate}, it is also possible to choose $\gamma$ to be the composition of two flows, the first for time $s$ and the second for time $s^2$: $\gamma_{x,t}(s)=\Phi_{s^2}(\ol{a}) \circ \Phi_s(b)(x)$. However, this requires solving two ODEs, whilst the choices of $\gamma$ above require only one.

\subsection{Regularity Considerations}

We now address the circumstances under which the regularity properties hold for $\gamma$. We focus our discussion on the $(\delta W)^2$-jet scheme for brevity. \Cref{assumptionLipschitz} is enough to give us good control over $\gamma(x,t,v)$. However, the regularity properties require control over the $v$-\textit{derivatives} of $\gamma$. We first remark that the case where $|v|$ is large presents no difficulty because one could (although we do not) redefine $\phi$ to be the flow of 
\begin{equation}
X^{\text{modified}}(x,t,v) = \sum_{\alpha=1}^k \left(v^{\alpha}b_{\alpha}(x,t) + \frac{1}{k} \sum_{\alpha=1}^k (v^{\alpha})^2 \: \ol{a}(x,t) \right) \Gamma(x,v)
\label{eq:PhiModifiedDefinition}
\end{equation} 
where $\Gamma$ is a smooth cutoff function equal to zero when $v>\min(1,e^{-x})$. This does not affect the derivatives of $\gamma$ at $v=0$. The most important factor is how the $x$-derivatives of $a(x,t)$ and $b(x,t)$ decay when $|x|$ is large. When these derivatives are zero outside a compact set, we are able to prove that $\gamma$ has the required regularity:

\begin{lemma}
Suppose that $a(x,t)$ and $b_\alpha(x,t)$ satisfy \cref{assumptionLipschitz,assumptionMeasurable}, are $(r+1)$-times differentiable, and are uniformly compactly supported. Then \cref{propertyLipschitzDerivs,propertyrthDerivative} (for all $r$) hold for both the $(\delta t)$ and $(\delta W)^2$-jet schemes.
\label{lemma:compactsupportregularity}
\end{lemma}

The proof is an iterative argument using an ODE comparison theorem; the details are postponed to the appendix. If we assume instead that the $x$-derivatives of $a$ and $b$ are uniformly bounded, then a similar ODE comparison argument yields bounds of the form
\begin{equation}
|\partial^{\alpha}_v \gamma| \leq K(1+|x|^{r_{\alpha}})e^{K(1+|v|^2)}
\label{eq:boundwithpowersofx}
\end{equation}
where $r_{\alpha}=1$ if $|\alpha|=0$ or $|\alpha|=1$. Should the $x$-derivatives of $a$, $b$ fail to decay at all, then it need not be true that $r_{\alpha}=1$ for all $\alpha$, as the following example shows.

\begin{example}
Consider the one-dimensional case $n=k=1$ and set $a=0$. Then $\Phi$ is given by
\[ \frac{\partial}{\partial s} \Phi = v b(\Phi) \]
with initial condition $\Phi(x,v,0)=x$. For a suitable constant $\Phi_0$, let
\[ F(\Phi)=\int_{\Phi_0}^{\Phi} \frac{1}{b(\Phi)} d\Phi \]
and let $G$ denote the inverse to $F$ (in the region where this exists). By separating the variables we may solve explicitly for $\Phi$ and hence compute the $v$-derivatives in terms of $F$ and $G$. We get
\[ \Phi(x,v,s)=G(vs+F(x)), \]
\[ \Phi_v(x,v,s)=sb(\Phi), \]
\[ \Phi_{vv}(x,v,s)=s^2 b'(\Phi)b(\Phi), \]
\[ \Phi_{vvv}(x,v,s)=s^3 b''(\Phi)b^2(\Phi) + s^3(b'(\Phi))^2b(\Phi). \]
Suppose that $b(x)=1+x+\sin(x)$. Then $b$ is Lipschitz with bounded derivatives. But
\[ \gamma_{vvv}(x,0)= -\sin(x) b^2(x) + (\cos(x))^2 b(x) \]
so we cannot take $r=1$ in \cref{eq:boundwithpowersofx} even when $v=0$.
\end{example}


However, we do not expect such pathologies to occur in practice. Every quantity in a numerical simulation is bounded due to the finite capacity of a computer, meaning that the assumptions of \cref{lemma:compactsupportregularity} apply. In situations where \cref{propertyLipschitzDerivs,propertyrthDerivative} do not hold, we can consider convergence in probability instead, a particularly appropriate notion for problems on manifolds because it is invariant under diffeomorphisms. We give an example to illustrate our approach; see also Appendix D of \cite{armstrong2018coordinate}. Suppose that $a$ and the $b_\alpha$ are known only to satisfy \Cref{assumptionMeasurable,assumptionLipschitz} and to be smooth. Then $\mathbb{E}[\max_{[0,T]} |X_t|] < \infty$, so we may choose a compact set $S$ such that the probability of $X$ leaving $S$ is arbitrarily small. Then \cref{propertyLipschitzDerivs,propertyrthDerivative} will hold for all $x,y \in S$. Our proof of convergence when the coefficients of the SDE are compactly supported then shows that our scheme converges in probability. Convergence in probability is metrisable via $d(X,Y)=\mathbb{E}[\min(|X-Y|,1)]$ for random variables $X,Y$, and so one can use our results to study the rate of convergence in this metric. Certain problems, most famously stochastic Lorenz-type systems \cite{geurts2019lyapunov} have \textit{attractors} to which the solution trajectories will, with high probability, be close at most late times. For such problems we would expect to prove stronger notions of convergence, but defer this to future work.

\section{Invariance of the Jet Scheme} \label{sec:invariance}

Fix integers $n$ and $k$. Let $\mathcal{S}$ be the set of SDEs driven by $k$-dimensional Brownian motion on (a chart of) a manifold $M$ with $n$ local coordinates. In other words, $\mathcal{S}$ consists of tuples $a(x,t), b_1(x,t), \dots, b_k(x,t)$ where \[ a: \mathbb{R}^n \times \mathbb{R} \rightarrow \mathbb{R}^n, \] \[ b_\alpha: \mathbb{R}^n \times \mathbb{R} \rightarrow \mathbb{R}^n \] for $\alpha=1,2,\dots,k$, where we assume for simplicity that $a$ and the $b_\alpha$ are smooth.

Let $\Gamma$ be the set of fields of maps $\gamma(x, \delta t, \delta W)$ where \[ \gamma: \mathbb{R}^n \times \mathbb{R} \times \mathbb{R}^{k} \rightarrow \mathbb{R}^n. \]

A numerical scheme may be viewed as converting an SDE to a difference equation, which can then be solved to give a simulation of the SDE. More formally, 

\begin{definition} A \textit{numerical scheme} is a function $N: \mathcal{S} \rightarrow \Gamma$. \end{definition}

Now let $f: M \rightarrow M'$ be a diffeomorphism of manifolds. Then $f$ acts on vector fields on $M$ via the pushforward $f_{*}$. Writing this in local coordinates, we obtain an action $f_{*}: \mathcal{S} \rightarrow \mathcal{S}$. This action is given by the usual chain rule if $s \in \mathcal{S}$ is expressed in Stratonovich calculus - see Proposition 1.2.4 of \cite{hsu2002stochastic}. If It\^{o} calculus is used, then the action is given by It\^{o}'s lemma: \begin{align*} f_{*} a(x,t)^i &= \sum_{j=1}^n \frac{\partial f^i}{\partial x^j} a^j(x,t) + \frac{1}{2} \sum_{j,k=1}^n \sum_{\alpha=1}^k \frac{\partial^2 f^i}{\partial x_j \partial x_k} b^j_\alpha (x,t) b^k_\alpha (x,t), \\ f_{*} b_{\alpha}(x,t)^i &= \sum_{j=1}^n \sum_{\alpha=1}^k \frac{\partial f^i}{\partial x^j}b^j_\alpha (x,t). \end{align*}

We also get an action of $f$ on $\Gamma$ by \[ f \gamma(x,t,v) = f \circ \gamma(f^{-1}(x),t,v). \]

\begin{definition} We say that a numerical scheme $N$ is \textit{invariantly defined} if \[ f N(a,b_1, \dots, b_k) = N(f_{*}(a,b_1,\dots,b_k)) \] for every diffeomorphism $f: M \rightarrow M'$. \end{definition}

In other words, to say that a scheme is invariantly defined is to say that the diagram \[ \begin{tikzcd} \text{SDE on } M \arrow{r}{\text{scheme}} \arrow[swap]{d}{\text{It\^{o}'s lemma } } & \{ \gamma_x : x \in M \} \arrow{d}{\text{action of } f} \\ \text{SDE on } M' \arrow{r}{\text{scheme}} & \{ \gamma_y: y \in M' \} \end{tikzcd} \] commutes. The concept of being invariantly defined, as explained above, is a special case of the general definition of invariantly-defined elements given in terms of category theory, as explained in \cite{armstrong2018markowitz}. 

\begin{example} \label{example:eulernotinvariant} We show that the Euler-Maruyama scheme is \textbf{not} invariantly defined. Applied to the SDE $\ed X = X \ed W$ on $M=(0,\infty)$, the E-M scheme selects the field of maps $\gamma_x(\delta t, \delta W)=x+x \: \delta W$, which, under the transformation $y=\log(x)$, transforms to $\gamma_y(\delta t, \delta W)=\log(e^y + e^y \delta W)$. This is not the same as $\hat{\gamma}_y(\delta t, \delta W) = y - \frac{1}{2} \delta t + \delta W$ which is the result of applying the E-M scheme to the transformed SDE $ \ed Y = -\frac{1}{2} \ed t + \ed W$. \end{example}

\begin{theorem} \label{thm:transformscorrectly} In the case $\tilde{\gamma}=\gamma$, both the $(\delta W)^2$ and $(\delta t)$-jet scheme are invariantly defined. \end{theorem} \begin{proof} Let $\mathfrak{X}(M)$ denote the set of vector fields on $M$. We say that a map \[ V: \mathcal{S} \rightarrow \mathfrak{X}(M) \] is an invariantly-defined vector field if \[ f_{*} V(a,b_1,\dots,b_k) = V(f_{*} (a,b_1,\dots,b_k)) \] Since Stratonovich SDEs transform via the usual chain rule, it follows that $\ol{a}$ and the $b_\alpha$ are invariantly defined. The jet scheme works by computing the flow of a linear combination of $\ol{a}$ and the $b_\alpha$. Since invariantly-defined operations on invariantly-defined objects always result in invariantly-defined output, the result follows. \end{proof} 

Whilst we cannot approximate $\gamma$ perfectly, \cref{thm:transformscorrectly} provides a genuine benefit; it introduces geometric invariance, even though the original SDE \cref{eq:sdeIto} was defined in the ambient space $\mathbb{R}^n$ and not (a priori) in a coordinate-free manner. If the step size of the scheme is small, then $\delta t$ and $\delta W$ will both be small with high probability, and hence for \textit{any} reasonable high-order ODE scheme, the approximation $\tilde{\gamma}$ will have negligible error. Compared to ODEs, it is particularly important to avoid a large step size if one desires strong accuracy, since the true solution $X_t$ may depend on $W_s$ for \textit{all} $0 \leq s \leq t$. But in cases where a large step size is necessary, the reader might consider a so-called \textit{aromatic} ODE method. These methods, designed to be equivariant under certain classes of diffeomorphisms, and hence `almost' invariantly defined, are the subject of recent research; see (e.g.) \cite{munte16aromatic}.

A practical consequence of \cref{thm:transformscorrectly} is that jet schemes will perform well on any problem which is reduced to a simple form by a judicious choice of coordinates. For example, if $F: \mathbb{R} \rightarrow \mathbb{R}$ is a smooth invertible function with inverse $G$, and $\mu, \sigma$ are constants, then the $(\delta t)$-jet scheme (with $\tilde{\gamma}=\gamma$) will simulate the SDE \begin{equation} \label{eq:disguisedlineareqn} \ed Y_t = (\mu F'(G(Y_t)) + \frac{1}{2}\sigma^2 F''(G(Y_t))) \ed t + F'(G(Y_t))\sigma \ed W \end{equation} perfectly, because substituting $Y=F(X)$ leads to the trivial SDE $\ed X_t = \mu \ed t + \sigma \ed W_t$. Informally, the $(\delta t)$-jet scheme can `see' the required substitution, whilst schemes based purely on truncating the stochastic Taylor series for an SDE will not be invariantly defined, and hence will fail to do so, unless $F$ is suitably chosen. Equations of the type \cref{eq:disguisedlineareqn} arise in practice: one example is geometric Brownian motion, which arises in the Black-Scholes model in financial mathematics \cite{black1973pricing} \cite{merton1969lifetime}. More generally, if an SDE is equivalent to a Brownian motion on $\mathbb{R}^n$ via a diffeomorphism then the $(\delta t)$-jet scheme will be exact.

\section{Application to the Kepler Problem with Noise} \label{sec:examples} Consider a particle moving in $\mathbb{R}^2$ under the influence of a single force directed towards the origin. Take coordinates $(r,p,\theta, \phi)$, where $p=\dot{r}$ and $\phi=\dot{\theta}$. Then the Lagrangian of this system is given by $L = \frac{1}{2}(p^2 + r^2 \phi^2) + V(r)$ where negative $V$ is the potential. We add noise to the system, obtaining dynamics expressed in It\^{o} form as \begin{align} \label{eq:keplersystem1} \ed r &= \left(p + \frac{1}{2}\xi_1(r) \xi_1'(r) \right) \: dt + \xi_1(r) \: \ed W_t^1, \\ \label{eq:keplersystem2} \ed p &= (-V'(r) + r \phi^2) \: \ed t, \\ \label{eq:keplersystem3} \ed \theta &= \phi \: \ed t + \xi_2(r) \: \ed W_t^2, \\ \label{eq:keplersystem4} \ed\phi &= \left(\frac{-2\phi p - \phi \xi_1(r) \xi_1'(r)}{r}+\frac{3\phi \xi_1^2(r)}{r^2} \right) \: dt - \frac{2 \phi \xi_1(r)}{r} \: \ed W_t^1. \end{align}

Here, $\xi_1$ and $\xi_2$ are user-chosen functions that specify the amount of noise in the model. It is straightforward to show that the angular momentum $h = r^2 \dot{\theta}$ is conserved in this system, and that if the $\xi_i$ are identically zero, then we recover the classical Kepler dynamics. The Stratonovich drift (for use in the jet scheme) is \[ \ol{a}(r,p,\theta,\phi) = \left[p, \: -V'(r) + r \phi^2, \: \phi, \: \frac{-2\phi p}{r}\right]^T \]

We motivate the system \cref{eq:keplersystem1}-\cref{eq:keplersystem4} with the following formal calculation. Taking the Legendre transform of $L$ yields the Hamiltonian $H = \frac{1}{2}p^2 + \frac{1}{2}\frac{h^2}{2r^2} - V(r)$. We perturb the generalised momenta in $H$, writing \begin{equation*} \ed H_{\text{perturbed}} = H \: \ed t + p \: \xi_1 \circ dW^1_t + h \: \xi_2 \circ dW^2_t. \end{equation*} Applying Hamilton's equations and translating back to $(r,p,\theta,\phi)$ coordinates yields the dynamics above. The above procedure is an example of stochastic advection by lie transport (SALT), a concept originating in fluid dynamics and intended to preserve the physics of the underlying system. For more details, see \cite{drivas2019lagrangian}. 

In a numerical experiment, we chose $V(r)=-\frac{1}{r}$, representing a planet moving under the influence of the sun. We simulated a Brownian path, and then found numerical solutions to the system \cref{eq:keplersystem1}-\cref{eq:keplersystem4} using first the Euler-Maruyama (E--M) scheme, and then the $(\delta t)$-jet scheme, with an 8th order Adams method to solve the ODE. The left and right halves of Figure 1 respectively show a plot of $r$ against $\theta$ for the E--M and jet schemes in the case $\xi_1=0.05$, $\xi_2 = 0.25$ (for all $r$), $(r_0,p_0,\theta_0,\phi_0)=(1,0.2,1,1.2)$, over a time period of $[0,10]$. A variety of step sizes were chosen, as indicated by the labels on the plot. For a very small step size $\delta t = 0.01$, the two schemes give very similar answers. However, as the step size increases, the E--M scheme starts giving qualitatively incorrect answers, whilst the accuracy of the jet scheme degrades only very slightly, with the four trajectories finishing in nearly the same place.

\begin{figure}[h!tbp] \label{figure:keplerproblem} \begin{center} \includegraphics[width=\textwidth]{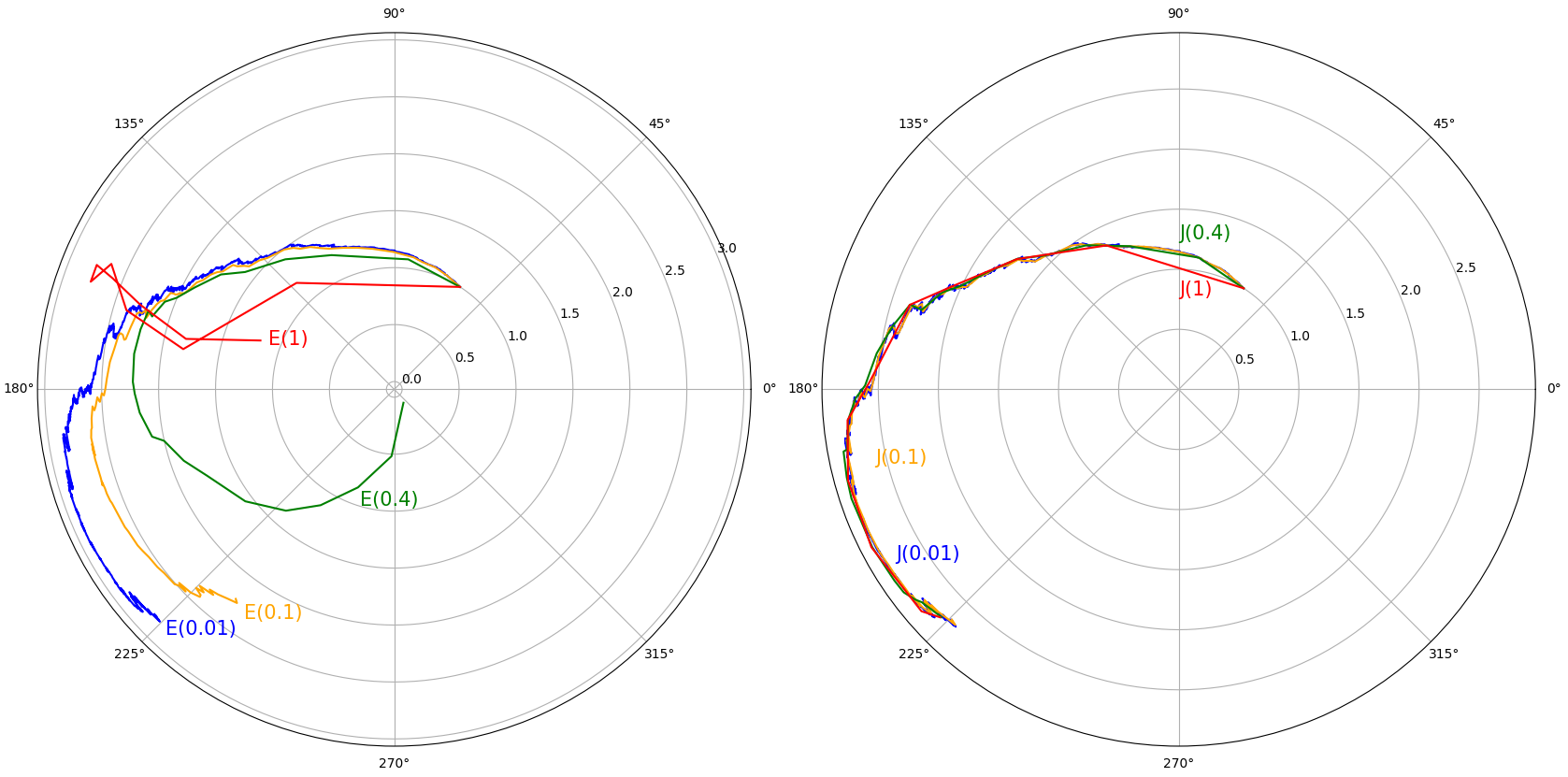} \end{center} \caption{Simulation of the Kepler problem with noise. On the left, the paths of the Euler scheme with $\delta t=1,0.4,0.1,0.01$, labelled $E(1)$, $E(0.4)$, $E(0.1)$ and $E(0.01)$ respectively. On the right, the jet scheme with the same step lengths, labelled similarly.} \end{figure}

We also tracked the value of the angular momentum $h$ estimated by the two schemes. At early times, both schemes calculated a value of $h$ close to the correct value $h=1.2$. But by the time we reach $t=10$, only the jet scheme maintains the correct value, as shown in Table 1.

\begin{table}[h!tbp] \label{table:kepler} \centering \begin{tabular}{lllc} \hline Step length & Scheme & Estimate of $h$ at $t=10$ & \\ \hline 1 & Euler--Maruyama & -0.022 \\ & Jet & 1.200004 \\ 0.4 & Euler--Maruyama & 0.023 \\ & Jet & 1.20001 \\ 0.1& Euler--Maruyama & 1.145 \\ & Jet & 1.20001 \\ 0.01 & Euler--Maruyama & 1.212 \\ & Jet & 1.200001 \\ \hline \end{tabular} \caption{Comparing the value of the angular momentum $h$ for the Euler and jet schemes.} \end{table}%

The very good performance of the jet scheme in this example is partially explained by the fact that since $\xi_1$ and $\xi_2$ were chosen to be constant, the system \cref{eq:keplersystem1}-\cref{eq:keplersystem4} is diffeomorphic (via $h=r^2 \phi$) to an SDE with additive noise. We modified the experiment so that this would no longer be the case by choosing $\xi_1(r)=0.05(1+\sin^2(r))$ and $\xi_2(r)=0.25(1+\sin^2(r))$. The plots in this case (not shown) of $r$ against $\theta$ were qualitatively similar to Figure 1. As we increased the step size for the jet scheme, the trajectories separated only to a slightly greater extent than in Figure 1. By contrast the E--M scheme again gave qualitatively incorrect answers for larger step sizes, and failed to fully converge even with $\delta t = 0.01$. This provides evidence for our suggestion in the introduction that the jet scheme should perform well on problems which are small perturbations of a more simple system.

\section{A General Result on the Convergence of Numerical Schemes} \label{sec:workhorse} Our aim in the next few sections is to prove \cref{thm:convergenceOnManifolds}. We begin by introducing $O_{\mathcal{F}}$ notation. This will allow us to write out proofs more cleanly, avoiding the need to write out minor variations of the same argument several times.

\begin{definition}[$O_{\cal F}$ notation] Let $\Delta_T=\{(t_1,t_2) : 0 \leq t_1 \leq t_2 \leq T \}$, and let $(\Omega,\mathcal{F},(\mathcal{F}_t: t \in [0,T]),\mathbb{P})$ be a filtered probability space. Suppose that we have two families of maps \[ f(t_1,t_2,\omega): \Delta_T \times \Omega \rightarrow \mathbb{R}^n, \] \[ g(t_1,t_2,\omega): \Delta_T \times \Omega \rightarrow \mathbb{R}^n \] such that, for each $t_1,t_2$, both $f(t_1,t_2,\cdot)$ and $g(t_1,t_2,\cdot)$ are random variables $\Omega \rightarrow \mathbb{R}^n$. We shall say that $f=O_{\mathcal{F}}(g)$ if, for every $q \in \mathbb{N}$ there exist constants $\epsilon_q$ and $C_q$ with the following property. For all $0 \leq t_1 \leq t_2 \leq T$ such that $t_2-t_1 \leq \epsilon_q$, we have \[ \mathbb{E}[ \: |f|^{2q} \mid \mathcal{F}_{t_1}] \leq C_q \mathbb{E}[ \: |g|^{2q} \mid \mathcal{F}_{t_1}] \] almost surely. If, in addition, $f$ satisfies $\mathbb{E}[f(t_1,t_2,\omega) \mid \mathcal{F}_{t_1}]=0$ for all $t_1,t_2$, we shall write $f=O^0_{\mathcal{F}}(g)$. \end{definition} \begin{example} If $a$ and $b$ are Lipschitz and $(Y_t: t \in [0,T])$ is a stochastic process adapted to $\mathcal{F}_t$ then, for a partition $0=t_0 \leq t_1 \leq \dots \leq t_N=T$ \[ a(Y_{t_{i}})(t_{i+1}-t_{i}) = O_\mathcal{F}(Y_{t_i}(t_{i+1}-t_{i})), \] \[ b(Y_{t_{i}})(W_{t_{i+1}}-W_{t_{i}}) = O^0_\mathcal{F}(Y_{t_i}(t_{i+1}-t_{i})^{1/2}). \] In what follows, we may write expressions such as $a(Y)\delta t = O(Y \delta t)$ and $b(Y)=O_\mathcal{F}^0(Y (\delta t)^{1/2})$ for brevity. \end{example}

\cref{theorem:strongworkhorse}, stated below, provides a sufficient condition for two numerical schemes to be close in the $\mathcal{L}^p$ sense. The proof of this result, relying on a combination of Gr\"{o}nwall's inequality and properties of martingales, is deferred to the appendix. The proof is similar to known arguments in \cite{kloeden2013numerical}, but encapsulates the essential details therein into a single reusable theorem. The result should be applicable to a wide class of potential numerical schemes; we demonstrate in the next two sections how \cref{theorem:strongworkhorse} may be applied repeatedly to deduce the convergence of the jet scheme. \begin{theorem} \label{theorem:strongworkhorse} Let $f$ and $\ol{f}$ be two functions \[ f,\ol{f}:\Delta_T \times \mathbb{R}^n \times \Omega \to \mathbb{R}^n. \] Let $0=t_0 \leq t_1 \leq \ldots \leq t_N=T$ be a discretization of $[0,T]$. Define $\delta t = \max_{0\leq i < N} \{ \tau_{i+1}-\tau_{i} \}$. Given $Y_0$ define sequences $Y_i$ and ${\ol Y}_i$ by $Y_0=\ol{Y}_0$, \[ Y_i(\omega) = Y_{i-1}(\omega) + f(\tau_{i-1}, \tau_i, Y_{i-1}, \omega ) \] and \[ {\ol Y}_i(\omega) = {\ol Y}_{i-1}(\omega) + \ol{f}(\tau_{i-1}, \tau_i, {\ol Y}_{i-1}, \omega ). \] Suppose that \begin{equation} f(t_1,t_2,Y,\omega)-\ol{f}(t_1,t_2,Y,\omega)=O_{\cal F}( (1 + |Y|) (\delta t)^{\gamma+1}) + O^0_{\cal F}( (1 + |Y|) (\delta t)^{\gamma+\frac{1}{2}}) \label{eq:yclosetoybar} \end{equation} for some $\gamma \geq 0$ and that \begin{equation} \ol{f}(t_1,t_2,Y,\omega)-\ol{f}(t_1,t_2,\ol{Y},\omega)=O_{\cal F}( (Y-\ol{Y}) (\delta t) ) + O^0_{\cal F}( (Y-\ol{Y}) (\delta t)^{\frac{1}{2}} ) \label{eq:lipschitsfbar} \end{equation} and that either \begin{equation} E( \max_{0\leq i \leq N} |Y_i|^{2q} ) \leq K \label{eq:ySquaredBound} \end{equation} or \begin{equation} E( \max_{0\leq i \leq N} |\ol{Y}_i|^{2q} ) \leq K. \label{eq:olYSquaredBound} \end{equation} Then \[ E\left(\max_{0\leq i \leq N} |Y_i-\ol{Y_i}|^{2q} \right) \leq K (\delta \tau)^{2 \gamma q} \] with $K$ a constant independent of the discretization $t_i$. \end{theorem}

\section{Proof of Strong Convergence} \label{sec:strongconvergence} In this section we apply \cref{theorem:strongworkhorse} repeatedly in order to prove the strong convergence result \cref{eq:mainresstrong}.

We introduce the notation for this section. Let $(W_t: t \in [0,T])$ be a Brownian motion in $\mathbb{R}^k$, and write $W_t^\alpha$ for its one-dimensional components. Let $(\mathcal{F}_t: t \in [0,T])$ be the natural filtration to which $W_t$ is adapted. We define a number of functions \[ f^*(t_1,t_2,Y,\omega): \Delta_T \times \mathbb{R}^n \times \Omega \rightarrow \mathbb{R}^n. \] These are \[ f^E(t_1,t_2, y, \omega) = a(x,t) \delta t + b_{\alpha}(x,t) \, \delta W^\alpha_t, \] \[ f^\gamma(t_1,t_2, y, \omega) = \gamma(y,t,\delta W_t)-y, \] \[ f^{\tilde\gamma}(t_1,t_2, y, \omega) = \tilde{\gamma}(y,t,\delta W_t)-y, \] \[ f^{(m)}(t_1,t_2, y, \omega) = {\gamma_m}(y,t,\delta W_t)-y, \] \[ f^{c}(t_1,t_2, y, \omega) = 0. \] As usual, $\delta t = t_2-t_1$ and $\delta W_t = W_{t_2}-W_{t_1}$.

Now let $0=t_0 \leq t_1 \leq \dots \leq t_N=T$ be a discretisation of $[0,T]$. Associated to each $f^\star$ and $Y_0 \in \mathbb{R}^n$, we define sequences of random variable $(Y^\star_i)_{i=1}^N$ by \[ Y_{i}^\star = Y_{i-1}^\star + f^\star (t_{i-1},t_{i},Y_{i-1},\omega) \] with initial condition $Y^\star_0=Y_0$. Thus the $Y^\star$ are the iterates of the Euler--Maruyama scheme, the jet scheme with a perfect ODE solver, the jet scheme with an imperfect ODE solver, the order $r$ expansion jet scheme, and the `constant scheme' $Y^c_i=Y_0$.

\begin{proposition} \label{prop:tildegammaforexpansionscheme} Suppose that $\gamma$ satisfies \cref{propertyrthDerivative} for some particular value $r=r_0+1$. Then the order $r_0$ expansion jet scheme is $(r_0+1)$-good. \end{proposition} \begin{proof} Fix $x \in \mathbb{R}^n$ and $t \in [0,T]$, and let $v \in \mathbb{R}^k$ vary. Each component $\gamma^i$ and $\gamma_m^i$ of the vectors $\gamma(x,t,v)$ and $\gamma_{r_0}(x,t,v)$ is then a function from $\mathbb{R}^k$ to $\mathbb{R}$. Therefore we may apply the multivariate Taylor expansion with Lagrange remainder to obtain \[ \gamma^i - \gamma_{r_0}^i = \sum_{|\alpha|=r_0+1} v^\alpha \frac{1}{\alpha!} \partial^\alpha_v \gamma^i(\xi v) \] for some $\xi \in [0,1]$. Hence \[ |\gamma^i - \gamma_{r_0}^i| \leq K| \sum_{\alpha} v^\alpha \partial^\alpha_v \gamma^i(\xi v)| \leq K|v|^{r_0}| \sum_{\alpha} \partial^\alpha_v \gamma^i(\xi v)|. \] \cref{propertyrthDerivative} then gives the required bound on the components of $\gamma-\gamma_{r_0}$, and hence on the vector itself. \end{proof}

We need an elementary lemma about the normal distribution; see the appendix for a proof.

\begin{lemma} \label{lemma:expectationofnormal} If $v \in \mathbb{R}^k$ is a multivariate normal vector with mean 0 and covariance matrix $I(\delta t)$ then for any $a>0,K \geq 0$ there exists constants $\epsilon_{a,K,k}$, $C_{a,K,k}$ for which \[ \mathbb{E}[|v|^a e^{K(1+|v|^2)} ] \leq C_{a,K,k} (\delta t)^{a/2} \hspace{10pt} (*) \] whenever $\delta t \leq \epsilon_{a,K,k}$. \end{lemma} 

\begin{proposition} \label{prop:ofbounds} Let $Y_t$ and $\ol{Y}_t$ be processes adapted to $\mathcal{F}_t$. Then: \begin{enumerate}[(i)] \item Provided that $\tilde{\gamma}$ is $m$-good, we have \[ f^\gamma(t_1,t_2,Y_{t_1},\omega) - f^{\tilde \gamma}(t_1,t_2,Y_{t_1},\omega) =O_\mathcal{F}((1+|Y_{t_1}|)(\delta t)^{m/2}). \] \item If $\gamma$ has \cref{propertyrthDerivative} for a particular value $r=r_0$ then \[ f^{\gamma}(t_1,t_2,Y_{t_1},\omega) - f^{(r_0)}(t_1,t_2,Y_{t_1},\omega) = O_{\mathcal{F}}((1+|Y_{t_1}|)(\delta t)^{\frac{r_0}{2}}). \] \item If $\gamma$ has \cref{propertyFirstTwoDerivs,propertyrthDerivative} when $r=2$, the order $2$ expansion jet scheme satisfies \[ f^E(t_1,t_2,Y_{t_1},\omega)-f^{(2)}(t_1,t_2,Y_{t_1},\omega) = O^0_{\mathcal{F}}((1+|Y|)(\delta t)^{1}) + O_{\mathcal{F}}((1+|Y|)(\delta t)^{\frac{3}{2}}). \] \item The Euler scheme satisfies \[ f^E(t_1,t_2,Y,\omega)-f^E(t_1,t_2,\ol{Y}) = O_\mathcal{F}(|Y_{t_1}-\ol{Y}_{t_1}|(\delta t)) + O^0_\mathcal{F}(|Y_{t_1}-\ol{Y}_{t_1}|(\delta t)^{1/2}). \] \item If $\gamma$ has \cref{propertyLipschitzDerivs}, then \[ f^{\gamma}(t_1,t_2,Y,\omega)-f^{\gamma}(t_1,t_2,\ol{Y},\omega) = O^0_\mathcal{F}(|Y_{t_1}-\ol{Y}_{t_1}|(\delta t)^{1/2}) + O_\mathcal{F}(|Y_{t_1}-\ol{Y}_{t_1}|(\delta t)). \] \end{enumerate} \end{proposition} 

\begin{proof} \begin{enumerate}[(i)] \item This is an immediate consequence of \cref{lemma:expectationofnormal}. \item Follows from part (i) and \cref{prop:tildegammaforexpansionscheme}. \item We have \begin{align*} f^E-f^{(2)} &= b_\alpha \: \delta W^\alpha + a \: \delta t - \big( \sum_{\alpha=1}^k \gamma_{v_\alpha} \: \delta W^\alpha + \gamma_{v_0} \: \delta t \\ &+ \sum_{\alpha=1}^k \frac{1}{2} \gamma_{v_\alpha v_\alpha} \: \delta t +\frac{1}{2} \sum_{\alpha=1}^k \gamma_{v_\alpha v_\alpha}((\delta W^\alpha)^2-\delta t) \\ &+\frac{1}{2} \sum_{\substack{\alpha, \beta = 1 \\ \alpha \neq \beta}}^k \gamma_{v_\alpha v_\beta} \: \delta W^\alpha \delta W^\beta + \frac{1}{2} \gamma_{v_0 v_0} (\delta t)^2 + \sum_{\alpha=1}^k \gamma_{v_0 v_\alpha} \delta t \: \delta W^\alpha \big), \end{align*} where all derivatives are evaluated at $v=0$. Using \cref{propertyFirstTwoDerivs} this simplifies to \begin{align*} f^E-f^{(2)} &= \frac{1}{2} \sum_{\alpha=1}^k \gamma_{v_\alpha v_\alpha}((\delta W^\alpha)^2-\delta t) +\frac{1}{2} \sum_{\substack{\alpha, \beta = 1 \\ \alpha \neq \beta}}^k \gamma_{v_\alpha v_\beta} \: \delta W^\alpha \delta W^\beta \\ &+ \frac{1}{2} \gamma_{v_0 v_0} (\delta t)^2 + \sum_{\alpha=1}^k \gamma_{v_0 v_\alpha} \delta t \: \delta W^\alpha. \end{align*} Using \cref{propertyrthDerivative} the first two terms are $O^0_\mathcal{F}((1+|Y|)(\delta t)$ and the last two are $O_\mathcal{F}((1+|Y|)(\delta t)^{3/2}$, as required. \item Follows from \cref{lemma:expectationofnormal} together with the Lipschitz properties of $a(x,t)$, $b(x,t)$. \item As in \cref{prop:tildegammaforexpansionscheme}, we may expand the components of $\gamma(x,t,v)-\gamma(y,t,v)$ with respect to $v$ \begin{align*} \gamma^i(x,t,v)-\gamma^i(y,t,v)-x^i+y^i &= v^0 \partial_{v_0}(\gamma^i(x,t,0)-\gamma^i(y,t,0)) \\ &+ \sum_{\alpha=1}^k v^\alpha \partial_{v_\alpha}(\gamma^i(x,t,0)-\gamma^i(y,t,0)) \\ &+ \sum_{\alpha,\beta=1}^k \frac{v^j v^k}{2} (\partial^2_{v_j v_k}(\gamma^i(x,t,\xi v)-\gamma^i(y,t,\xi v)). \end{align*} By \cref{propertyLipschitzDerivs}, the first term on the right is $O_\mathcal{F}(|x-y|(\delta t))$ and the second is $O^0_\mathcal{F}(|x-y|(\delta t)^{1/2})$. For the last term, the condition $0\leq \xi \leq 1$ and \cref{propertyLipschitzDerivs} imply that \begin{equation*} |\frac{v^j v^k}{2} (\partial^2_{v_j v_k}(\gamma^i(x,t,\xi v)-\gamma^i(y,t,\xi v))| \leq K|v|^2e^{K(1+|v|^2)}|x-y| \end{equation*} and now \cref{lemma:expectationofnormal} implies that this is $O_{\mathcal{F}}(|x-y|(\delta t))$. \end{enumerate} \end{proof}

\begin{proposition} \label{prop:boundedmoments} Suppose that $\gamma$ has \cref{propertyrthDerivative} for $r=2$. Suppose also that $\tilde{\gamma}$ is $m$-good for some $m \geq 2$. Then \[ \mathbb{E}[\max_{0 \leq i \leq N} |Y_i^\star|^{2q} ] \leq K \] where $\star$ can be any of $E,\gamma,\tilde{\gamma}$. \end{proposition} \begin{proof} This is an immediate consequence of \cref{theorem:strongworkhorse} when $\gamma=0$, $Y_i=Y^\star_i$ and $\ol{Y}_i=Y^c_i$, so we verify that the assumptions of this theorem are satisfied. \cref{eq:lipschitsfbar,eq:olYSquaredBound} trivially hold. So it remains to check \cref{eq:yclosetoybar}, namely that \[ f^\star(t_1,t_2,Y,\omega)-f^c(t_1,t_2,Y,\omega)=O_{\cal F}( (1 + |Y|) (\delta t)) + O^0_{\cal F}( (1 + |Y|) (\delta t)^{\frac{1}{2}}). \] This is easily seen to hold for the Euler scheme or the order 2 expansion jet scheme (under \cref{propertyrthDerivative}). For the case where $\star$ is either $\gamma$ or $\tilde{\gamma}$, use parts (i) and (ii) of \cref{prop:ofbounds}, together with the triangle inequality. \end{proof}

\begin{proof}[\textbf{Proof of \cref{thm:convergenceOnManifolds}} (Part I)] We prove that the jet scheme converges to the true solution with strong order $\frac{1}{2}$, and also establish \cref{eq:mainresstrong}. The remainder of \cref{thm:convergenceOnManifolds} is proved in \cref{sec:weakconvergence}. Since $a(x,t)$ and $b(x,t)$ are Lipschitz, the Euler scheme $Y^E_i$ is known to converge to the solution of \cref{eq:sdeIto} with strong order $\frac{1}{2}$, so it suffices to show that \[ \mathbb{E}[\max_{1 \leq i \leq N} | Y_i^{\tilde{\gamma}} - Y_i^E|^2 ] \leq C(\delta t). \] Therefore, we seek to apply \cref{theorem:strongworkhorse} when $\gamma=\frac{1}{2}$, $Y_i=Y_i^{\tilde{\gamma}}$ and $\ol{Y}_i = Y^E_i$. \cref{prop:boundedmoments,prop:ofbounds} (iv) respectively show that assumptions \cref{eq:olYSquaredBound,eq:lipschitsfbar} of the Theorem hold. Parts (i) (ii) and (iii) of \cref{prop:ofbounds}, together with the triangle inequality, show that \cref{eq:yclosetoybar} holds. To prove \cref{eq:mainresstrong}, apply \cref{theorem:strongworkhorse} when $\gamma=\frac{m}{2}-1$, $Y_i=Y^{\tilde{\gamma}}_i$ and $\ol{Y}_i=Y^{\gamma}_i$. \cref{prop:ofbounds,prop:boundedmoments} respectively show that \cref{eq:yclosetoybar} and \cref{eq:olYSquaredBound} hold. Finally, \cref{prop:ofbounds} (v) shows \cref{eq:lipschitsfbar}. \end{proof}

\section{Proof of Weak Convergence} \label{sec:weakconvergence}

We borrow some notation from \cite{kloeden2013numerical}. Let $P_l=\{1,2,\dots,n \}^l$, and for $y \in \mathbb{R}^n$ and $p \in P_l$, define \[ F_p(y) = \prod_{h=1}^l y^{p_h}. \]

\Cref{theorem:weakworkhorse}, stated below, is a weak analogue of \cref{theorem:strongworkhorse}. The proof is found in the appendix.

\begin{theorem} \label{theorem:weakworkhorse} Let $\gamma$ and $\ol{\gamma}$ be two functions \[ \gamma,\ol{\gamma}:\Delta_T \times \mathbb{R}^n \times \Omega \to \mathbb{R}^n. \] Let $0=t_0 \leq t_1 \leq t_2 \leq \ldots \leq t_N=T$ be a discretization of $[0,T]$. Define $\delta t = \max_{0\leq i < N} \{ t_{i+1}-t_{i} \}$. Given $Y_0=\ol{Y}_0$ define sequences $Y_i$ and ${\ol Y}_i$ by \[ Y_i(\omega) = \gamma(t_{i-1}, t_i, Y_{i-1}, \omega ) \] and \[ {\ol Y}_i(\omega) = \ol{\gamma}(t_{i-1}, t_i, {\ol Y}_{i-1}, \omega ). \] Suppose that for every $q \in \mathbb{N}$ there exists constants $K$, $r$ and $\beta$ such that \begin{equation} \label{eq:weakworkhorseA1} \mathbb{E}[|\ol{Y}_i|^{2q}] \leq K(1+|Y_0|^{2r}) \end{equation} \begin{equation} \label{eq:weakworkhorseA2} \mathbb{E}[|\ol{Y}_{i}-\ol{Y}_{i-1}|^{2q} \mid \mathcal{F}_{i-1}] \leq K(1+\max_{j} |\ol{Y}_j|^{r})(t_{i}-t_{i-1})^q \end{equation} \begin{equation} \label{eq:weakworkhorseA3} \mathbb{E}[|\gamma(\ol{Y}_{i-1})-\ol{Y}_{i-1}|^{2q} \mid \mathcal{F}_{i-1}] \leq K(1+\max_{j} |\ol{Y}_j|^{r})(t_{i}-t_{i-1})^q \end{equation} and \begin{equation} \begin{split} \label{eq:weakworkhorseMainAssumption} \mathbb{E}[F_p(\ol{Y}_i-\ol{Y}_{i-1})-F_p(\gamma(\ol{Y}_{i-1})-\ol{Y}_{i-1}) \mid \mathcal{F}_{i-1}] \\ \leq K(1+\max_{i}|\ol{Y}_i|^{2r})(\delta t)^{\beta}(t_{i}-t_{i-1}). \end{split} \end{equation} Then for any smooth function $g: \mathbb{R}^n \rightarrow \mathbb{R}$ of at most polynomial growth, there exist constants $K'$ and $r'$ such that \[ \mathbb{E}[g(Y_N)]-\mathbb{E}[g(\ol{Y}_N)] \leq K'(1+|Y_0|^{r'})(\delta t)^{\beta}. \] \end{theorem}

\begin{lemma} Let $\gamma$ satisfy \cref{propertyrthDerivative} for $r=3$. Let $\tilde{\gamma}$ be an $m$-good approximation to $\gamma$, with $m \geq 2$. Then for any $y \in \mathbb{R}^n$ we have \begin{equation*} \mathbb{E}[|\tilde{\gamma}(y,\delta t, \delta W_t)-y|^{2q}] \leq K(1+y^r)(\delta t)^q. \end{equation*} \label{lemma:weakassumptioncheckA} \end{lemma} \begin{proof} We first check that the inequality holds when $\tilde{\gamma}$ is the order 2 expansion jet scheme. All derivatives $\partial_v \gamma(y,\delta t, \delta W)$ are evaluated at $(y,0,0)$, and we use the inequality $(a_1+\dots+a_n)^{2q} \leq K_{n,q} \sum |a_i|^{2q}$. We obtain \begin{align*} \mathbb{E}[|\gamma^{(2)}-y|^{2q}] &= \mathbb{E}[|\partial_{v_0} \gamma \: \delta t + \sum_{i=1}^k \partial_{v_i}\gamma \: \delta W^i + \partial^2_{v_0} \gamma \: (\delta t)^2 \\ &\hspace{10pt} + \sum_{i=1}^k \partial^2_{v_0 v_i} (\delta t)(\delta W^i) + \sum_{i,j=1}^k \partial^2_{v_i v_j}\gamma \: \delta W^i \delta W^j |^{2q} ] \\ &\leq K(1+y^r) \sum_{i,j=1}^k \mathbb{E}[|\delta t|^{2q} + |\delta W^i |^{2q} |\delta t|^{2q} + |\delta W^i|^{2q}+|\delta W^i \delta W^j|^{2q}] \\ &\leq K(1+y^r)(\delta t)^q. \end{align*} To deduce the lemma, recall (\cref{prop:tildegammaforexpansionscheme}) that $\gamma^{(2)}$ is 3-good (and hence 2-good) under \cref{propertyrthDerivative}. Hence \begin{align*} \mathbb{E}[|\tilde{\gamma}-y|^{2q}] &\leq K \mathbb{E}[|\tilde{\gamma}-\gamma^{(2)}|^{2q}] + K \mathbb{E}[|\gamma^{(2)}-y |^{2q}] \\ &\leq K(1+y^r)\mathbb{E}[|\delta W|^{4q}e^{Kq(1+|\delta W|^2)}]+K(1+y^r)(\delta t)^q \\ &\leq K(1+y^r)(\delta t)^q \end{align*} as required, where in the final line we used Lemma \ref{lemma:expectationofnormal}. \end{proof}

\begin{lemma} Adopt the notation of Theorem \ref{theorem:weakworkhorse}. Assume that $\gamma$ has property \ref{propertyFirstTwoDerivs} and \ref{propertyrthDerivative} for $r=2$. Assume also that \[ \mathbb{E}[\gamma(y, \delta t, \delta W_t)^{2q}] \text{ and } \mathbb{E}[\ol{\gamma}(y, \delta t, \delta W_t)^{2q}] \leq K(1+y^r). \] \begin{enumerate}[(i)] \item If $\gamma=\gamma^E$ and $\ol{\gamma}=\gamma^{(2)}$ then Assumption \cref{eq:weakworkhorseMainAssumption} of Theorem \ref{theorem:weakworkhorse} holds for $\beta=1$. \item If $\gamma=\gamma$ (the exact jet scheme) and $\ol{\gamma}=\gamma^{(2)}$ then \eqref{eq:weakworkhorseMainAssumption} holds for $\beta=1$. \item If $\gamma=\gamma$ and $\ol{\gamma}=\tilde{\gamma}$ where $\tilde{\gamma}$ is $m$-good then \eqref{eq:weakworkhorseMainAssumption} holds for $\beta=\frac{m}{2}$. \end{enumerate} \label{lemma:weakAssumptionCheckB} \end{lemma} \begin{proof} \begin{enumerate}[(i)] \item We have \[ \gamma^E(y, \delta t, \delta W) - y = b_\alpha \delta W^\alpha + a \: \delta t, \] \begin{equation*} \begin{split} \gamma^{(2)}(y, \delta t, \delta W) - y = \partial_{v_0} \gamma \: \delta t + \sum_{i=1}^k \partial_{v_i} \gamma \: \delta W^i \\ + \frac{1}{2} \partial^2_{v_0 v_0} \gamma \: (\delta t)^2 + \frac{1}{2}\sum_{i,j=1}^k \partial_{ij} \gamma \: \delta W^i \delta W^j + \sum_{i=1}^k \partial^2_{v_0 v_i} \gamma \: \delta t \delta W^i. \end{split} \end{equation*} For some vector $p \in P_l$, expand the terms in \begin{equation*} \mathbb{E}[F_p(\ol{Y}_i-\ol{Y}_{i-1})-F_p(\gamma(\ol{Y}_{i-1})-\ol{Y}_{i-1}) \mid \mathcal{F}_{i-1}]. \end{equation*} Since we are aiming for $\beta=1$, we can neglect any terms of order $(dt)^2$ or higher. For this reason, we may assume that $l=1$ or $l=2$. Moreover, any terms with an odd number of $\delta W$s will vanish when expectations are taken. Therefore it suffices to show that, under expectation, the terms with exactly two $dW$s cancel out. This follows from \cref{propertyFirstTwoDerivs}. \item This is an immediate consequence of (iii). (Alternatively one can use a direct computation similar to part (i).) \item Let $v$ and $w$ be vectors in $\mathbb{R}^n$, and let $e_1,\dots,e_n$ be integers. Then \[v_1^{e_1}\dots v_n^{e_n}-w_1^{e_1}\dots w_n^{e_n} = (v_1^{e_1}-w_1^{e_1})(v_2^{e_2}\dots v_n^{e_n}) +w_1^{e_1}(v_2^{e_n}\dots v_n^{e_n}-w_2^{e_2} \dots w_n^{e_n}). \] Since $v-w$ is a factor of $v^e-w^e$ for every $e$, it follows by induction on $n$ that $v_1^{e_1}\dots v_n^{e_n}-w_1^{e_1}\dots w_n^{e_n}$ may be written as a sum of polynomials $P_i$, each of which has $(v_i-w_i)$ as a factor for some $i$. Hence \begin{equation*} \begin{split} \mathbb{E}[F_p(\ol{\gamma}(y)-y)-F_p(\gamma(y)-y)]=\sum \mathbb{E}[P_i] \\ \leq K \sum \sqrt{\mathbb{E}[|\ol{\gamma}-\gamma|^2]} \sqrt{\mathbb{E}[|1+\ol{\gamma}^{r_i}+\gamma^{r_i}|^2]}. \end{split} \end{equation*} Since $\ol{\gamma}$ is $m$-good, we obtain $\mathbb{E}[\sqrt{|\ol{\gamma}-\gamma|^2}] \leq (\delta t)^m(1+y^r)$. The result now follows from the bounds on $\mathbb{E}[\gamma(y,dW_t)^{2q}]$ and $\mathbb{E}[\ol{\gamma}(y,dW_t)^{2q}]$. \end{enumerate} \end{proof} 

\begin{proof}[\textbf{Proof of \cref{thm:convergenceOnManifolds} (Part II)}] First apply \cref{theorem:weakworkhorse} with $\gamma=\gamma^E$ and $\gamma=\gamma^{(2)}$. Proposition \ref{prop:boundedmoments} tells us that Assumption \eqref{eq:weakworkhorseA1} holds. Assumptions \eqref{eq:weakworkhorseA2} and \eqref{eq:weakworkhorseA3} follow from Lemma \ref{lemma:weakassumptioncheckA}, and \eqref{eq:weakworkhorseMainAssumption} follows from Lemma \ref{lemma:weakAssumptionCheckB}. Next apply the same theorem with $\gamma=\gamma^{(2)}$ and $\gamma=\tilde{\gamma}$. This time, \eqref{eq:weakworkhorseMainAssumption} follows from parts (ii) and (iii) of Lemma \ref{lemma:weakAssumptionCheckB}, together with the triangle inequality. (Alternatively, $\gamma^{(2)}$ may be viewed as a $3$-good approximation not only to $\gamma$, but to $\tilde{\gamma}$ as well, so the same argument applies.) Finally, apply the same theorem with $\gamma=\gamma$ and $\ol{\gamma}=\tilde{\gamma}$. \end{proof} 

\section{Conclusions} \label{sec:conclusions} Given a physical system modelled by some SDE, we introduced a class of numerical schemes for the system which automatically preserve any of its constraints to high order. Our approach has two advantages over projection approaches. First, the manifold does not need to be known. Second, if the SDE is merely concentrated on $M$, our scheme may still be applied. We established the convergence of our schemes under a standard set of assumptions in both the strong and weak sense. We also applied them to a stochastic version of the Kepler problem, in which the scheme not only preserved the angular momentum constraint, but gave a much better approximation overall than the Euler--Maruyama scheme.

The $(\delta t)$-version of our scheme performs essentially perfectly on any SDE diffeomorphic to $n$-dimensional Brownian motion, so may be expected to give good results for any SDE which is close to such an SDE. It is not necessary to know the diffeomorphism explicitly. We therefore believe that the invariance properties of jet schemes may well be beneficial in problems without constraints, and we will explore this in future research.

\section*{Acknowledgments} The authors thank Erwin Luesink and Alex Mijatovi\'{c} for helpful discussions.  This work was supported by the Engineering and Physical Sciences Research Council [EP/L015234/1], The EPSRC Centre for Doctoral Training in Geometry and Number Theory (The London School of Geometry and Number Theory), University College London. Both authors are members of the Department of Mathematics at King's College London, and thank the same for its support.

\begin{appendices}

\section{Properties of the jet map}

\begin{proof}[Proof of \cref{lemma:gammadef}] 
Since $\ol{a}$ and the $b_{\alpha}$ are tangent to $M$, so is any linear combination of them. We deduce that $\phi$ is the flow of a tangent vector field to $M$, and hence that \cref{propertyRemainsInM} holds. It remains to check \cref{propertyFirstTwoDerivs}. We use the Einstein summation convention in this proof, and for brevity we suppress $t$ from our notation. We have
\begin{equation}
\Phi = x + \int_{0}^s \left( v^\alpha b_\alpha(\Phi)+ \frac{1}{k} v^r v^r \: \overline{a}(\Phi) \: \right) dz. \label{eq:Phidef}
\end{equation}
Setting $v=0$ immediately gives $\Phi(x,0,s)=x$. Differentiating, we obtain
\begin{equation}
\frac{\partial \Phi}{\partial v^\alpha} = \int_{0}^s \Bigl( b_{\alpha}(\Phi)+v^r \frac{\partial b_r}{\partial{x_i}}\rvert_{\Phi} \frac{\partial \Phi^i}{\partial v^\alpha} + \frac{2}{k} v^\alpha \overline{a}(\Phi) + \frac{1}{k}v^r v^r \frac{\partial \overline{a}}{\partial x^i}\rvert_{\Phi} \frac{\partial \Phi^i}{\partial v^\alpha} \Bigr) \: dz.
\label{eq:PhiFirstDeriv}
\end{equation}
Setting $v=0$ gives
\begin{equation*}
\frac{\partial \Phi}{\partial v^\alpha}(x,0,s) = \int_{0}^s b_{\alpha}(\Phi(x,0,s)) \:dz =sb_{\alpha}(x).
\end{equation*}
Differentiating once more gives
\begin{align*} \frac{\partial^2 \Phi}{\partial^2 v^\alpha v^\beta} &= \int_0^s \Bigl(\frac{\partial b_\alpha}{\partial{x_i}}\rvert_{\Phi}\frac{\partial \Phi^i}{\partial v^\beta} \\ 
&+ \frac{\partial b_{\beta}}{\partial x_i}\frac{\partial \Phi^i}{\partial v^\alpha} + v^r \frac{\partial \Phi^i}{\partial v^\alpha}\frac{\partial \Phi^j}{\partial v^\beta}\frac{\partial^2 b_r}{\partial x_i \partial x_j}\rvert_{\Phi} + v^r \frac{\partial b_r}{\partial x_i} \frac{\partial^2 \phi^i}{\partial v^\alpha \partial v^\beta} \\
&+ \frac{2}{k}\delta_{\alpha \beta} \overline{a}(\Phi)+\frac{2}{k}v^\alpha \frac{\partial \overline{a}}{\partial x^i}\rvert_{\Phi} \frac{\partial \Phi^i}{\partial v^\beta} + \\
&+ \frac{2}{k}v^\beta \frac{\partial \overline{a}}{\partial x^i}\rvert_{\Phi} \frac{\partial \phi^i}{\partial v^\alpha} + \frac{1}{k}v^r v^r \frac{\partial \Phi^i}{\partial v^\alpha}\frac{\partial \Phi^j}{\partial v^\beta} \frac{\partial^2 \overline{a}}{\partial x^i \partial x^j}\rvert_{\Phi}+\frac{1}{k}v^r v^r \frac{\partial \overline{a}}{\partial x^i} \rvert_{\Phi} \frac{\partial^2 \Phi^i}{\partial v^\alpha \partial v^\beta} \Bigr) \: dz. \label{eq:PhiSecondDeriv} \end{align*}
Setting $v=0$ and $s=1$ gives
\begin{align*}
\frac{\partial \Phi}{\partial^2 v^\alpha v^\beta}(0) &= \int_0^1 \Bigl( \frac{\partial b_\alpha}{\partial{x_i}}\rvert_{\Phi}\frac{\partial \Phi^i}{\partial v^\beta}(0) + \frac{\partial b_{\beta}}{\partial x_i}\frac{\partial \Phi^i}{\partial v^\alpha}(0) + \frac{2}{k}\delta_{\alpha \beta} \overline{a}(\Phi) \Bigr) \: dz \\ 
&= \int_0^1 \Bigl( \frac{\partial b_\alpha}{\partial x_i}|_{\Phi}b^i_\beta(x)z +\frac{\partial b_{\beta}}{\partial x_i}b^i_\alpha(x)z +\frac{2}{k}\ol{a}(x)\delta_{\alpha\beta} \Bigr) \: dz \\
&=\frac{1}{2}\frac{\partial b_\alpha}{\partial x_i}|_{\Phi}b^i_\beta(x) +\frac{1}{2}\frac{\partial b_{\beta}}{\partial x_i}b^i_\alpha(x) +\frac{2}{k}\ol{a}(x)\delta_{\alpha\beta}. 
\end{align*}
Finally we compute the Laplacian
\begin{align*}
\frac{\partial \Phi}{\partial^2 v^\alpha v^\alpha}(0) &= \frac{\partial b_\alpha}{\partial x_i}|_{\Phi}b^i_\alpha(x)+2\ol{a}(x) \\ 
&=2a(x).
\end{align*}
This establishes all but the final statement of the result, which may be proved in the same manner as above.
\end{proof}

To prove \cref{lemma:compactsupportregularity} we need the following ODE comparison theorem.

\begin{lemma}
\label{lemma:ODEcomparison}
If $u(s): [0,\infty) \rightarrow \mathbb{R}^n$ is a differentiable function such that $|\dot{u}| \leq A|u|+B$ for some constants $A>0$, $B \geq 0$, and $|u(0)| \leq C$, then $|u(s)| \leq (C+\frac{B}{A})e^{As}-\frac{B}{A}$ for all $s$. \end{lemma}
\begin{proof} The solution of the ODE $\dot{y}=Ay+B$ with initial condition $y(0)=C$ is $y(s)=(C+\frac{B}{A})e^{As}-\frac{B}{A}$. Since the function $y \mapsto Ay+B$ is Lipschitz, the claim follows from Theorem D.2 (ODE Comparison) of \cite{lee2013smooth} 
\end{proof}

In the following, the constant $K$ may not depend upon $v,s,t$ or $x$, but is allowed to change its value from line to line. 

\begin{proof}[Proof of \cref{lemma:compactsupportregularity}]
We give the proof for the $(\delta W)^2$-jet scheme; the proof for the $(\delta t)$-jet scheme is similar. Let $\Phi$ be as defined in \cref{lemma:gammadef}, and let a dot denote partial differentiation with respect to $s$. From the Lipschitz properties of $a$ and $b$ we obtain
\[ |\dot{\Phi}| \leq K(|v|+|v|^2) + K(|v|+|v|^2)|\Phi| \]
and so \cref{lemma:ODEcomparison} gives
\[ |\Phi(x,t,v,s)| \leq (1+|x|)e^{Ks(1+|v|^2)}. \]
From \cref{eq:PhiFirstDeriv} we have
\begin{align*}
|{\partial_{v_\alpha}}\dot{\Phi}| &\leq K+K|v| |\partial_{v_\alpha}\Phi |+K|v| +K|v|^2 |\partial_{v_\alpha} \Phi | \\
&\leq K(1+|v|^2)+K(1+|v|^2)|\partial_{v_\alpha} \Phi|, 
\end{align*}
so \cref{lemma:ODEcomparison} gives 
\[ |\Phi(x,t,v,s)| \leq K e^{Ks(1+|v|^2)}. \]
The choice of $C=0$ in applying \cref{lemma:ODEcomparison} is justified because \cref{eq:PhiDefinition} tells us that $\Phi(x,t,v,s)=x$ when $s=0$. Differentiating \cref{eq:PhiFirstDeriv} repeatedly, and using an inductive argument, it follows that every $v$-derivative of $\Phi$ satisfies the same bound. This proves \cref{propertyrthDerivative}. Finally, applying a similar argument involving derivatives such as $\partial^2_{x v_i} \gamma$ gives 
\[ \left|\frac{\partial^3 \Phi}{\partial v^\alpha v^\beta \partial x^i} \right| \leq K(1+|v|^2)e^{Ks(1+|v|^2)} \] \label{example:compactlysupported} which implies \cref{propertyLipschitzDerivs}.
\end{proof} 

\section{Strong Convergence}

We now seek to prove \cref{theorem:strongworkhorse}, which requires introducing some preliminary lemmas. Recall that $q$ is a nonnegative integer and $K$ is a constant which may depend on $q$ but not on $x$, $v$ or $t$. The value of $K$ may change from line to line.

\begin{lemma}
Given a partition $0=t_0 \leq t_1 \leq \dots \leq t_N=T$, suppose we have families of random variables $X(t_1,t_2,\omega)$ and $Y(t_1,t_2,\omega)$ such that $X=O_\mathcal{F}(Y)$. Write $X_i$ to mean $X(t_{i},t_{i+1},\omega)$ and similarly for $Y$. Then
\[\mathbb{E}[ \max_{0 \leq i < N}|\sum_{j=0}^i X_j|^{2q}] \leq K N^{2q-1} \mathbb{E}[ \sum_{i=0}^{N-1} |Y_{i}|^{2q}]. \]
If, in addition, $X=O^0_\mathcal{F}(Y)$ then 
\[\mathbb{E}[ \max_{0 \leq i < N}|\sum_{j=0}^i X_j|^{2q}] \leq K N^{q-1} \mathbb{E}[ \sum_{i=0}^{N-1} |Y_{i}|^{2q} ]. \] 
\label{lemma:boundsumincrements2}
\end{lemma}

\begin{proof}
We first show that if $a_i, i=1,2,\dots,2q$ is a collection of vectors then
\begin{equation}
\label{vectorsinequality}
\left| \sum_{r=1}^N a_r \right|^{2q} \leq \sum_{\alpha_1,\dots,\alpha_{2q}=0}^{N} \sum_{r=1}^{2q} |a_{\alpha_r}|^{2q}.
\end{equation}
Note that
\[ \left| \sum_{r=1}^N a_r \right|^{2q}= \left( \sum_{\alpha_1,\beta_1=1}^N \langle a_{\alpha_1}, a_{\beta_1} \rangle \right)^q =\sum^N_{\alpha_1,\dots,\alpha_q, \beta_1,\dots,\beta_q}\langle a_{\alpha_1}, a_{\beta_1} \rangle\langle a_{\alpha_2}, a_{\beta_2} \rangle\dots\langle a_{\alpha_q}, a_{\beta_q} \rangle. \]
Each $\langle a_{\alpha_i}, a_{\beta_j} \rangle$ is at most $\max_i |a_i|^2$ and hence the product of these brackets is at most $\max_i |a_i|^{2q}$. Since the sum on the RHS of \cref{vectorsinequality} includes this maximum term, we have proved our claim. Let us now prove the lemma itself. Let $Q$ denote the quantity we wish to bound,
\[ Q:=E\left[ \max_{0 \leq i < N} \left| \sum_{k=0}^i X_k \right|^{2q} \right]. \]
Using \cref{vectorsinequality} we obtain
\[ Q \leq E\left[ \max_{0 \leq i < N} \sum_{\alpha_1,\alpha_2,\dots,\alpha_{2q}=0}^i \sum_{j=0}^{2q} |X_{\alpha_j}|^{2q} \right]. \]
Since all the terms in this sum are non-negative, we may eliminate the $\max$ from our expression, obtaining 
\[ Q \leq E\left[ \sum_{\alpha_1,\alpha_2,\dots,\alpha_{2q}=0}^{N-1} \sum_{j=0}^{2q} |X_{\alpha_j}|^{2q} \right]. \]
This expression is symmetric in the $\alpha_i$ and hence 
\[ Q \leq K E\left[ \sum_{\alpha_1,\alpha_2,\dots,\alpha_{2q}=0}^{N-1} |X_{\alpha_1}|^{2q} \right]= K N^{2q-1} \mathbb{E}\left[ \sum_{\alpha_1=0}^{N-1} |X_{\alpha_1}|^{2q} \right]. \]
Since $X$ is $O_{\mathcal{F}}(Y)$ we obtain 
\begin{align*} Q \leq K N^{2q-1} \mathbb{E}\left[\mathbb{E}\left[ \sum_{i=0}^{N-1} |X_{i}|^{2q} \mid \mathcal{F}_{t_i} \right] \right] &\leq K N^{2q-1} \mathbb{E}\left[\mathbb{E}\left[ \sum_{i=0}^{N-1} |Y_{i}|^{2q} \mid \mathcal{F}_{t_i} \right] \right] \\
&=K N^{2q-1} \mathbb{E}\left[ \sum_{i=0}^{N-1} |Y_{i}|^{2q} \right]
\end{align*} 
as required. Now suppose that $X$ is $O^0_{\mathcal{F}}(Y)$. Notice that $S_i:=\sum_{j=0}^i X_i$ is a martingale with respect to $\mathcal{F}_{t_i}$. Therefore, Doob's $\mathcal{L}^p$ inequality gives us 
\[ \mathbb{E}[\max_{0 \leq i < n} |S_i|^{2q}] \leq K \mathbb{E}[|S_N|^{2q}]. \] 
Consider the continuous time process $(S^c_t: t \in [0,T])$ given by $S^c_t=S_{i(t)}$, where $i(t)$ is the greatest integer $i$ such that $t_i \leq t$. The quadratic variation of this process at time $T$ is given by $[S^c]_T=\sum_{0 \leq i <N} |X_{t_i}|^2$. Hence, by the Burkholder-Davis-Gundy inequality we have 
\[ \mathbb{E}[|S_N|^{2q}] = \mathbb{E}[|S^c_T|^{2q}] \leq K \mathbb{E}[(|X_1|^2+\dots+|X_n|^2)^q]. \]
Arguing as before, it then follows that \begin{align*} \mathbb{E}[|S_N|^{2q}] &\leq K \sum_{\alpha_1 , \dots, \alpha_q = 0}^{N-1} |X_{\alpha_1}|^2\dots|X_{\alpha_q}|^2 \\ 
&\leq K \mathbb{E}[\sum_{\alpha_1,\dots,\alpha_q=0}^{N-1} \sum_{r=1}^q |X_{\alpha_r}|^{2q}] \\ 
&\leq K \mathbb{E}[\sum_{\alpha_1,\dots,\alpha_q=0}^{N-1} |X_{\alpha_1}|^{2q}].
\end{align*}
Thus
\[\mathbb{E}[|S_N|^{2q}] \leq K N^{q-1} \mathbb{E}[ \sum_{\alpha_1=0}^{N-1} |X_{\alpha}|^{2q} ] \leq K N^{q-1} \mathbb{E}[ \sum_{i=0}^{N-1} |Y_{i}|^{2q} ] \]
as required. 
\end{proof}

\begin{corollary} 
\label{corollary:boundsumincrements2}
Again, let $0=t_0 \leq t_1 \leq t_N=T$ be a partition. Write $(\delta t)_i=t_{i+1}-t_{i}$ and $\delta t = \max_{0 \leq i < N} (\delta t)_i$. Consider families of random variables of the form 
\[ Y_i=f(S(t_{i}),\omega)(\delta t)_i^{r} \]
where $r$ is real, $S(t,\Omega): [0,T] \times \Omega \rightarrow \mathbb{R}^n$ is a process adapted to $\mathcal{F}_t$ and $f: \mathbb{R}^n \times \Omega \rightarrow \mathbb{R}^n$ is a function. Let $\gamma \geq 0$ be real. Suppose either that $X=O_{\mathcal{F}}(Y)$ when $r=\gamma+1$ or that $X=O^0_{\mathcal{F}}(Y)$ when $r=\gamma+\frac{1}{2}$. Then, in either case, \cref{lemma:boundsumincrements2} gives the same bound for $\mathbb{E}[ \max_{0 \leq i < N}|\sum_{j=0}^i X_j|^{2q}]$, which we write as 
\begin{multline} \mathbb{E}\left( \max_{0 \leq i < N} \left| \sum_{k=0}^i O_{\cal F}\left((\delta t)^{\gamma+1} f(S) \right)[t_k,t_{k+1},S_k,\omega] \right|^{2q} \right) \\
\leq K (\delta t)^{2\gamma q} \sum_{i=0}^{N-1} \mathbb{E} \left( \max_{0\leq j \leq i} |f(S_j)|^{2q} \right) \delta t_i.
\end{multline}
and
\begin{multline}
\mathbb{E}\left( \max_{0 \leq i < N} \left| \sum_{k=0}^i O^0_{\cal F}\left((\delta t)^{\gamma+\frac{1}{2}} f(S) \right)[t_k,t_{k+1},S_k,\omega] \right|^{2q} \right) \\
\leq K (\delta t)^{2\gamma q} \sum_{i=0}^{N-1} \mathbb{E} \left( \max_{0\leq j \leq i} |f(S_j)|^{2q} \right) \delta t_i. 
\end{multline}
\end{corollary}
\begin{proof} 
For the first result, substituting the given expression for $Y$ into the first part of \cref{lemma:boundsumincrements2}, and using the fact that $N \leq K/(\delta t)$, gives
\begin{align*}
\mathbb{E}[ \max_{0 \leq i < N}|\sum_{j=0}^i X_j|^{2q}] &\leq K N^{2q-1} \mathbb{E}[ \sum_{i=0}^{N-1} |f(S_i)|^{2q}(\delta t)^{2q(\gamma+1)-1}(\delta t)_i ] \\
&\leq KN^{2q-1-(2q(\gamma+1)-1)} \mathbb{E}[ \sum_{i=0}^{N-1} |f(S_i)|^{2q}(\delta t)_i ] \\
&\leq K(\delta t)^{2q \gamma} \mathbb{E}[ \sum_{i=0}^{N-1} |f(S_i)|^{2q}(\delta t)_i ] 
\end{align*} which is the first result. The same argument using the second part of \cref{lemma:boundsumincrements2} gives the second result.
\end{proof} 

We shall need the following version of the discrete Gr\"{o}nwall lemma; the proof is in Proposition 1 of \cite{holte2009discrete}. 
\begin{proposition}
\label{prop:discretegronwall}
Let $y_n$, $f_n$ and $g_n$ be non-negative sequences. Suppose that
\begin{equation*} 
y_n \leq f_n + \sum_{0 \leq k <n} g_k y_k
\end{equation*} 
then
\begin{equation*}
y_n \leq f_n + \sum_{0 \leq k <n}f_k g_k \exp(\sum_{k<j<n}g_j)
\end{equation*}
\end{proposition} 

\begin{proof}[Proof of \cref{theorem:strongworkhorse}]
Given $t \in [0,T]$, let $n_t=\max\{i \in \{0,1,\ldots , N\}:t_i\leq t\}$. Define $H(t)$ by 
\begin{equation}
\begin{split} 
H(t)&=E\left( \max_{0 \leq i \leq n_t} |Y_i-\ol{Y_i}|^{2q} \right) \\
&=E\left( \max_{0 \leq i \leq n_t} \left| \sum_{k=0}^{i-1} f(\tau_k,\tau_{k+1}, Y_k, \omega ) - \sum_{k=0}^{i-1} \ol{f}(\tau_k,\tau_{k+1}, \ol{Y}_k, \omega ) \right|^{2q} \right) \\
&=E\left( \max_{0 \leq i \leq n_t} \left| \sum_{k=0}^{i-1} \left( f(\tau_k,\tau_{k+1}, Y_k, \omega ) - \ol{f}(\tau_k. \tau_{k+1}, Y_k, \omega ) \right) \right. \right. \\ &\qquad \qquad \qquad \left. \left. + \sum_{k=0}^{i-1} \left( \ol{f}(\tau_k,\tau_{k+1}, Y_k, \omega )-\ol{f}(\tau_k,\tau_{k+1}, \ol{Y}_k, \omega ) \right) \right|^{2q} \right). \\ 
\end{split}
\end{equation} We deduce that
\begin{equation}
\begin{split}
H(t) &= E\left( \max_{0 \leq i \leq n_t} \left| \sum_{k=0}^{i-1} \left( O_{\cal F}((1+|Y_i|)(\delta \tau)^{\gamma+1}) + O^0_{\cal F}((1+|Y_i|)(\delta \tau)^{\gamma+\frac{1}{2}} \right. \right. \right. \\
&\qquad \qquad \qquad + \left. \left. \left. O_{\cal F}(|Y_i-\ol{Y}_i|(\delta \tau)) + O^0_{\cal F}(|Y_i-\ol{Y}_i|(\delta \tau)^{\frac{1}{2}} \right) \right|^{2q} \right).
\end{split}
\end{equation}
Hence
\begin{equation}
\begin{split}
H(t) &\leq K \left\{ E\left( \max_{0 \leq i \leq n_t} \left| \sum_{k=0}^{i-1} O_{\cal F}((1+|Y_i|)( \delta t)^{\gamma+1} ) \right|^{2q} \right) \right. \\
&\qquad \qquad \left. + E\left( \max_{0 \leq i \leq n_t} \left| \sum_{k=0}^{i-1} O^0_{\cal F}((1+|Y_i|)( \delta t)^{\gamma+\frac{1}{2}} ) \right|^{2q} \right) \right. \\
&\qquad \qquad \left. + E\left( \max_{0 \leq i \leq n_t} \left| \sum_{k=0}^{i-1} O_{\cal F}(|Y_i-\ol{Y}_i|( \delta t) ) \right|^{2q} \right) \right. \\
&\qquad \qquad \left. + E\left( \max_{0 \leq i \leq n_t} \left| \sum_{k=0}^{i-1} O^0_{\cal F}(|Y_i-\ol{Y}_i|( \delta t)^{\frac{1}{2}} ) \right|^{2q} \right) \right\}.
\end{split}
\end{equation}
Using \cref{corollary:boundsumincrements2} we find
\begin{equation} \begin{split} H(t)&\leq K \left\{ (\delta t)^{2 \gamma q} \sum_{i=0}^{n_t-1} E\left( \max_{0 \leq j \leq i} \left\{ 1 + |Y_j|^{2q} \right\} \right) \delta t_i \right. \\
&\qquad \qquad \left. + \sum_{i=0}^{n_t-1} E\left( \max_{0 \leq j \leq i} \left| Y_j-\ol{Y}_j \right|^{2q} \right) \delta t_i \right\}.
\end{split}
\label{hBound}
\end{equation}
Assume that \cref{eq:ySquaredBound} holds. Then we have
\begin{equation}
H(t_n) \leq K (\delta t)^{2 \gamma q} + K\sum_{i=0}^{n-1} H(t_i) \delta t_i,
\end{equation} and hence, taking $y_n=H(t_n)$, $f_n=K(\delta t)^{2\gamma q}$ and $g_n=K(\delta t)$ in \cref{prop:discretegronwall}, we obtain
\begin{equation}
H(t_n) \leq K (\delta t)^{2 \gamma q} + K (\delta t)^{2 \gamma q} \sum_{0 \leq i <n} (\delta t) \exp(K(n-i) (\delta t)),
\end{equation}
from which we deduce the required result. \\ 
Assume instead that \cref{eq:olYSquaredBound} holds. Note that
\[ |Y_j|^{2q} = |Y_j - \ol{Y}_j + \ol{Y}_j|^{2q} \leq K(|Y_j-\ol{Y}_j|^{2q}+|\ol{Y}_j|^{2q}). \] 
Substituting this into \cref{hBound} reveals that 
\begin{equation}
\begin{split} 
H(t)&\leq K \left\{ (\delta t)^{2 \gamma q} \sum_{i=0}^{n_t-1} E\left( \max_{0 \leq j \leq i} \left\{ 1 + |\ol{Y}_j|^{2q} \right\} \right) \delta t_i \right. \\ &\qquad \qquad \left. + \sum_{i=0}^{n_t-1} E\left( \max_{0 \leq j \leq i} \left| Y_j-\ol{Y}_j \right|^{2q} \right) \delta t_i \right\} \\ 
\end{split} 
\end{equation} from which we may proceed as before. 
\end{proof} 

\begin{proof}[Proof of \cref{lemma:expectationofnormal}] 
Choose $\delta t$ sufficiently small that $\mathbb{E}[e^{Kv_i^2}] \leq 2$ for each $i$. Then
\begin{align*}
\mathbb{E}[(v_1^2 + \dots + v_k^2)^a e^{K(1+v_1^2+\dots+v_k^2)}] &\leq C_{a,K,k} \mathbb{E}\big[ \sum_{i=1}^k v_i^{2a} \prod_{j=1}^k e^{K v_j^2} \big] \\
& = C_{a,K,k} \sum_{i=1}^k \prod_{j \neq i} \mathbb{E}[e^{K v_j^2}] \: \mathbb{E}\big[ v_i^{2a}e^{v_i^2}\big] \\
&\leq C_{a,K,k} \mathbb{E}[v_1^a e^{Kv_1^2}],
\end{align*}
where the constant $C$ can change from line to line. It therefore suffices to prove the result when $k=1$. Provided $(\delta t)$ is sufficiently small, we have
\begin{align*}
\mathbb{E}[|v_1|^a e^{K v_1^2}] &= \frac{1}{\sqrt{2 \pi (\delta t)}} \int_{-\infty}^\infty |v_1|^{2a} \exp\left(Kv_1^2 - \frac{v_1^2}{2(\delta t)}\right) \: dv_1 \\
&\leq \frac{C_{a,K,k}}{\sqrt{(\delta t)}} \int_0^\infty v_1^{2a} \exp\left( - \frac{v_1^2}{4(\delta t)}\right) \: dv_1 \\
&= 4^a C_{a,K,k} \Gamma(a+\frac{1}{2}) (\delta t)^a, 
\end{align*} which proves the claim.
\end{proof}

\section{Weak Convergence}

\begin{proof}[Proof of \cref{theorem:weakworkhorse}]
We follow the argument in Theorem 14.5.2 of \cite{kloeden2013numerical}. Given $0 \leq m \leq N$ and $y \in \mathbb{R}^n$ define $(Y_{i}^{m,y})_{i=m}^N$ recursively by 
\[ Y_m^{m,y}=y, \]
\[ Y_i^{m,y}=\gamma(i-1,Y_{i-1}^{m,y}):=\gamma(t_{i-1},t_{i},Y_{i-1}^{m,y},\omega) \]
and define
\[ u(m,y)=\mathbb{E}[g(Y_N^{m,y})]. \]
In this notation, $Y_i$ and $\ol{Y}_i$ as defined in the statement of the Theorem are $Y_i^{0,Y_0}$ and $\ol{Y}_i^{0,\ol{Y}_0}$ respectively. Note that $Y_i^{m,\gamma(m-1,y,\omega)}=Y_i^{m-1,y}$ almost surely for all $m,y$ and $i \geq m$, so 
\begin{equation*}
u(m,\gamma(m-1,y,\omega)))=u(m-1,y).
\end{equation*}
In particular 
\begin{equation}
u(m,\gamma(m-1,\ol{Y}_{m-1})) =u (m-1,\ol{Y}_{m-1}).
\label{eq:utimechange}
\end{equation}
We seek to bound $H$, where
\begin{equation*}
H:= \mathbb{E}[g(\ol{Y}_N)] - \mathbb{E}[g(Y_N)].
\end{equation*} Using the fact that $Y_0=\ol{Y}_0$, and \cref{eq:utimechange}, we write
\begin{align*} 
H &= \mathbb{E}[u(N,\ol{Y}_N)-u(0,\ol{Y}_0)]| \\
&= \mathbb{E}[\sum_{i=1}^{N} u(i,\ol{Y}_i)-u(i-1,\ol{Y}_{i-1}) ] \\ 
&=\mathbb{E}[\sum_{i=1}^N u(i,\ol{Y}_i)-u(i,\gamma(i-1,\ol{Y}_{i-1},\omega))] \\
&=\mathbb{E}[\sum_{i=1}^N (u(i,\ol{Y}_i) - u(i,\ol{Y}_{i-1})) \: + \: \sum_{i=1}^N u(i,\ol{Y}_{i-1})-u(i,\gamma(i-1,\ol{Y}_{i-1},\omega))]. 
\end{align*}
We now Taylor expand in the second argument of $u$. For brevity, write $\gamma(\ol{Y}_{i-1})$ instead of $\gamma(i-1,\ol{Y}_{i-1},\omega)$. 
\begin{align*} 
H &= \sum_{i=1}^N \sum_{l=1}^{2\beta+1} \frac{1}{l!} \sum_{p \in P_l} \partial^p_y u(i,\ol{Y}_{i-1})F_p(\ol{Y}_i-\ol{Y}_{i-1}) + R_{i}(\ol{Y}_i) \\
&- \sum_{i=1}^N \sum_{l=1}^{2\beta+1} \frac{1}{l!} \sum_{p \in P_l} \partial^p_y u(i,\ol{Y}_{i-1})F_p(\gamma(\ol{Y}_{i-1})-\ol{Y}_{i-1}) + R_{i}(\gamma(\ol{Y}_{i-1})), 
\end{align*} 
where the remainder terms have the form 
\begin{equation*} 
R_{i}(Z) = \frac{1}{(2\beta+2)!} \sum_{p \in P_{2(\beta+1)}} \partial^p_y u(i,\ol{Y}_{i-1}+\theta_{p,i}(Z)(Z-\ol{Y}_{i-1})) F_p(Z-\ol{Y}_{i-1})
\end{equation*}
for $Z=\ol{Y}_i$ and $Z=\gamma(\ol{Y}_{i-1})$ respectively, and the entries in the diagonal matrix $\theta$ lie in $(0,1)$. We first bound the main term, and then deal with the remainder.
\small
\begin{align*} 
\mathbb{E}[H_{\text{main}}] &:= \mathbb{E} \left[\sum_{i=1}^N \sum_{l=1}^{2\beta+1} \frac{1}{l!} \sum_{p \in P_l} \partial^p_y u(i,\ol{Y}_{i-1}) \left(F_p(\ol{Y}_i-\ol{Y}_{i-1})-F_p(\gamma(\ol{Y}_{i-1})-\ol{Y}_{i-1})\right) \right] \\
&=\mathbb{E}\left[\mathbb{E} \left[ \sum_{i,l} \frac{1}{l!} \sum_{p \in P_l} \partial^p_y u(i,\ol{Y}_{i-1}) \left(F_p(\ol{Y}_i-\ol{Y}_{i-1})-F_p(\gamma(\ol{Y}_{i-1})-\ol{Y}_{i-1})\right) \mid \mathcal{F}_{i-1} \right] \right] \\
&=\sum_{i,l} \sum_{p \in P_l}\frac{1}{l!} \mathbb{E}\left[ \partial^p_y u(i,\ol{Y}_{i-1}) \mathbb{E}\left[ \left(F_p(\ol{Y}_i-\ol{Y}_{i-1})-F_p(\gamma(\ol{Y}_{i-1})-\ol{Y}_{i-1})\right) \mid \mathcal{F}_{i-1} \right] \right] 
\end{align*}
\normalsize
Using \cref{eq:weakworkhorseMainAssumption} we deduce (recall as usual that $r$ may change from line to line) \begin{align*}
\mathbb{E}[|H_{\text{main}}|] &\leq \sum_{i=1}^N \sum_{l=1}^{2\beta+1}\sum_{p \in P_l} K \mathbb{E}\left[ |\partial^p_y u(i,\ol{Y}_{i-1})| |1+\max_j (\ol{Y}_{j})^{2r}|(\delta t)^{\beta+1} \right] \\ &\leq \sum_{i=1}^N \sum_{l=1}^{2\beta+1}\sum_{p \in P_l} K \mathbb{E}\left[ |1+\max_j (\ol{Y}_{j})^{2r}|(\delta t)^{\beta+1} \right] \\
&\leq K \mathbb{E}[1+\max_j (\ol{Y}_j)^{2r}](\delta t)^{\beta}] \\
&\leq K(\delta t)^\beta(1+|Y_0|^r),
\end{align*} where for the last line we used \cref{eq:weakworkhorseA1}. The bound on $|\partial^p_y u(i,\ol{Y}_{i-1})|$ came from \cref{eq:weakworkhorseA1}, the definition of $u$ and the polynomial growth of $g$. By Cauchy-Schwarz the remainder term satisfies
\begin{multline*} \mathbb{E}[|R_{i}(Z)|] \leq \sum_{p \in P_l} \sqrt{\mathbb{E}[\mathbb{E}[ u(i,\ol{Y}_{i-1}+\theta_{p,i}(Z)(Z-\ol{Y}_{i-1}))^2 \mid \mathcal{F}_{i-1}]]} \\ \sqrt{\mathbb{E}[\mathbb{E}[ F_p(Z-\ol{Y}_{i-1})^2 \mid \mathcal{F}_{i-1}]]}.
\end{multline*} 
By \cref{eq:weakworkhorseA2} and \cref{eq:weakworkhorseA3} we deduce 
{ \small 
\begin{equation*} \mathbb{E}[|R_{i}(Z)|] \leq K \sqrt{\mathbb{E}[\mathbb{E}[ u(i,\ol{Y}_{i-1}+\theta_{p,i}(Z)(Z-\ol{Y}_{i-1}))^2 \mid \mathcal{F}_{i-1}]]} \sqrt{\mathbb{E}[1+\max_j (|\ol{Y}_j)|^{2r}] (\delta t)^{2\beta+2}}. 
\end{equation*} }
Finally, by the polynomial growth of $u$, together with \cref{eq:weakworkhorseA2} and \cref{eq:weakworkhorseA3} we get
\begin{equation*} 
\begin{split}
\mathbb{E}[|R_{i}(Z)|] \leq K \sqrt{\mathbb{E}[\mathbb{E}[ 1 + |\ol{Y}_{i-1}|^r + |Z-Y_{i-1}|^r \mid \mathcal{F}_{i-1} ]]} \sqrt{\mathbb{E}[1+\max_j (|\ol{Y}_j|)^{2r}] (\delta t)^{2\beta+2}} \\ \leq K \sqrt{1+\max_j |\ol{Y}_j|^{2r}} \sqrt{\mathbb{E}[1+\max_j (|\ol{Y}_j|)^{2r}]} \: (\delta t)^{\beta+1}. 
\end{split}
\end{equation*}
Hence 
\begin{equation*}
\sum_{i=1}^N \mathbb{E}[|R_i|] \leq K(1+\max_j |\ol{Y}_j|^r)(\delta t)^\beta, 
\end{equation*} as required.
\end{proof}

\end{appendices}

\bibliographystyle{plain}
\bibliography{jetscheme}

\end{document}